\documentclass{amsart}
\usepackage{wasysym}
\usepackage{amsmath}
\usepackage{amssymb}
\usepackage{amsthm}
\usepackage{bbm} 

\usepackage{tikz}
\usetikzlibrary{matrix,arrows,backgrounds,shapes.misc,shapes.geometric,patterns,calc,positioning}
\usetikzlibrary{calc,shapes}
\usepackage{wrapfig}
\usepackage{epsfig}  
\usepackage{mathtools}
\usepackage[margin=1.3in]{geometry}
\usepackage{color}	 
\usepackage{MnSymbol}
\usepackage{xy,color}
\xyoption{poly}
\xyoption{2cell}
\xyoption{all}
\usepackage{bm}
\newcommand{\quiversize}{11pt}
 
\newtheorem{thm}{Theorem}[section]
\newtheorem{prop}[thm]{Proposition}
\newtheorem{lemma}[thm]{Lemma}
\newtheorem{corollary}[thm]{Corollary}

\newtheorem*{thm*}{Theorem}

\usepackage{hyperref}
\theoremstyle{definition}
\newtheorem{definition}[thm]{Definition}
\newtheorem{construction}[thm]{Construction}
\theoremstyle{remark}
\newtheorem{remark}[thm]{Remark}

\numberwithin{equation}{section}

\newcommand{\calc}{\mathcal{C}}
\newcommand{\cals}{\mathcal{S}}

\newcommand{\cald}{\mathcal{D}}

\newcommand{\za}{\alpha}
\newcommand{\zb}{\beta}
\newcommand{\zd}{\delta}
\newcommand{\zD}{\Delta}
\newcommand{\ze}{\epsilon}
\newcommand{\zg}{\gamma}
\newcommand{\zG}{\Gamma}
\newcommand{\zl}{\lambda}
\newcommand{\zs}{\sigma}

\newcommand{\zO}{\Omega}
\newcommand{\kb}{\mathbbm{k}}
\newcommand{\ot}{\leftarrow}
\newcommand{\zebar}{\overline{\epsilon}}

\newcommand{\Hom}{\textup{Hom}}
\newcommand{\add}{\textup{add}}

\newcommand{\rad}{\textup{rad}\,}

\newcommand{\zabar}{\overline{\alpha}}
\newcommand{\zbbar}{\overline{\beta}}
\newcommand{\zgbar}{\overline{\gamma}}
\newcommand{\zdbar}{\overline{\delta}}
\newcommand{\zsbar}{\overline{\sigma}}

\newcommand{\Ext}{\textup{Ext}}
\newcommand{\End}{\textup{End}}
\newcommand{\cmp}{\textup{CMP}}
\newcommand{\scmp}{\underline{\cmp}}
\newcommand{\diag}{\textup{Diag}(\cals)}
\newcommand{\diags}{\textup{Diag}(S)}
\newcommand{\coker}{\textup{coker}}
\newcommand{\im}{\textup{im}}

\newcommand{\wt}{\textup{w}}
\newcommand{\cwt}{\overline{\textup{w}}}

\begin{document}

\title{A geometric model for syzygies over 2-Calabi-Yau tilted algebras II} 
\author{Ralf Schiffler}
\thanks{}
\address{Department of Mathematics, University of Connecticut, Storrs, CT 06269-1009, USA}
\email{schiffler@math.uconn.edu}

\author{Khrystyna Serhiyenko}
\thanks{} 
\address{Department of Mathematics, University of Kentucky, Lexington, KY 40506-0027, USA }
\email{khrystyna.serhiyenko@uky.edu}

\begin{abstract}
In this article, we continue the study of a certain family of 2-Calabi-Yau tilted algebras, called dimer tree algebras. The terminology comes from the fact that these algebras can also be realized as quotients of dimer algebras on a disc.  They are defined by a quiver with potential whose dual graph is a tree, and they are generally of wild representation type.  Given such an algebra $B$, we construct a polygon $\cals$ with a checkerboard pattern in its interior, that defines a category $\diag$.  The indecomposable objects of $\diag$ are the 2-diagonals in $\cals$, and its morphisms are certain pivoting moves between the 2-diagonals. We prove that the category $\diag$ is equivalent to the stable syzygy category of the algebra $B$.  This result was conjectured by the authors in an earlier paper, where it was proved in the special case where every chordless cycle is of length three.  

As a consequence, we conclude that the number of indecomposable syzygies is finite, and moreover the syzygy category is equivalent to the 2-cluster category of type $\mathbb{A}$.  In addition, we obtain an explicit description of the projective resolutions, which are periodic.  Finally, the number of vertices of the polygon $\cals$ is a derived invariant and a singular invariant for dimer tree algebras, which can be easily computed form the quiver. 
\end{abstract}

\maketitle

\begin{figure}[h]
\begin{center}
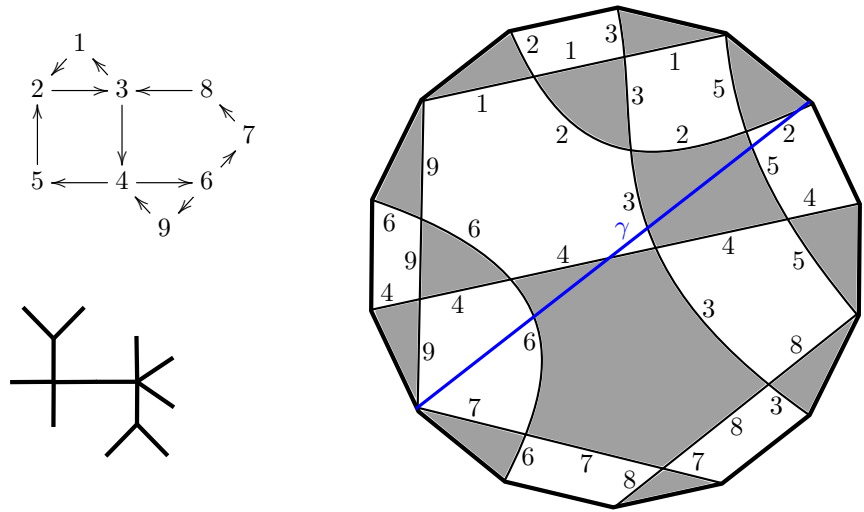
\caption{A quiver $Q$ together with its dual graph on the left, and the corresponding checkerboard polygon on the right. The module $M_{\zg}$ determined by the 2-diagonal $\zg$ is the cokernel of the map $f_{\zg}: P(5)\oplus P(6)\to P(3)\oplus P(4)$, determined by the crossing of $\zg$ with the radical lines 3,4,5, and 6.}
\label{fig:dual_graph}
\end{center}
\end{figure}

\section{Introduction}

In this paper, we continue our study of syzygies over 2-Calabi-Yau tilted algebras initiated in \cite{SS}.  Our main result is the proof of the main conjecture of that article in full generality.

The syzygy modules over a ring or an algebra play a fundamental role in both commutative and non-commutative algebra.  By definition, they are the submodules of projective modules, and hence syzygies are the building blocks of projective (or free) resolutions. In particular, every module can be approximated by its syzygies.  Understanding the category of syzygies provides valuable information about the algebra, but in general this is a difficult problem.  The algebras we consider here are wild in general, and thus there is no hope for understanding their module category entirely.  However, this paper provides a complete description of the syzygy category.

Our algebras have the property that their syzygy category is equivalent to the category of maximal Cohen-Macauley modules and also to the singularity category.   Cohen-Macauley modules and rings are central in commutative algebra, in particular, in the McKay correspondence, matrix factorization, and resolutions of singularities \cite{E, LW, M, PV, Y}.  Buchweitz brought these ideas to the non-commutative setting of Iwanaga-Gorenstein rings, introducing the singularity category in \cite{Bu}, which was later rediscovered and generalized to the graded setting by Orlov \cite{O}.  An important problem is the classification of rings of finite Cohen-Macauley type, which is solved for hypersurface singularities and for normal Cohen-Macauley rings of Krull dimension two, in the commutative case.  For higher dimensions, as well as for non-commutative rings the problem is open. 

In this paper, we are interested in a special class of 2-Calabi-Yau tilted algebras.  The family of 2-Calabi-Yau tilted algebras arises from the categorification of cluster algebras \cite{Amiot, BMRRT} and are a generalization of cluster-tilted algebras and of Jacobian algebras of quivers with potentials.  Keller and Reiten showed that every 2-Calabi-Yau tilted algebra is Iwanaga-Gorenstein of Gorenstein dimension one and that their stable syzygy category is 3-Calabi-Yau \cite{KR}. 

Here we study a special family of 2-Calabi-Yau tilted algebras which are characterized by the condition that the dual graph of their quiver is a tree, see Definition~\ref{def Q}.  The potential is given by the alternating sum of the chordless cycles.  These algebras can be realized as quotients of dimer algebras on a disc which implies that the boundary arrows in our quiver also induce relations on the algebra. These are zero relations and guarantee that the algebra is finite-dimensional and schurian.  Because of this similarity we call our algebras \emph{dimer tree algebras}.  
For example, algebras arising from the coordinate rings of the Grassmannians $\textup{Gr}(3,n)$ are dimer tree algebras. Dimer algebras have been studied extensively, see \cite{HK, Po,  JKS, BKM, Pr} and the references therein; for their connection to homological mirror symmetry, see \cite{Bocklandt}. 

\subsubsection*{Main result} 
Let $B$ be a dimer tree algebra with quiver $Q$. Let $\text{mod}\,B$ denote the category of finitely generated right $B$-modules and $\scmp\,B$ the category of non-projective syzygies.  Also, let $\Omega$ denote the syzygy functor.  To every boundary arrow $\alpha$ of $Q$, we associate a weight $\wt(\alpha)$ which equals either 1 or 2, depending on the parity of the length of its cycle path (or zigzag path), see Definition~\ref{def cycle path 2}.  The \emph{total weight} of $B$, which is the sum of weights of all boundary arrows will be denoted by $2N$.  We use a polygon $\mathcal{S}$ with $2N$ vertices to provide a geometric model for the syzygy category $\scmp\,B$ of $B$.  

Let $\diags$ be the category of 2-diagonals in $\mathcal{S}$, whose morphisms are given by pivoting moves between the 2-diagonals.  This is a triangulated category whose shift functor is given by the clockwise rotation $R$ by $\pi/N$.   Our polygon $\mathcal{S}$ is also quipped with a checkerboard pattern that is defined by a set of radical lines $\rho(i)$ associated to the vertices $i$ of $Q$. Then each 2-diagonal $\zg$ in $\mathcal{S}$ corresponds to an indecomposable syzygy $M_{\zg}$ and its intersections with the checkerboard pattern determines the projective presentation of $M_{\zg}$.   Moreover, the 2-diagonal $\zg$ will be oriented and hence its crossings with the lines $\rho(i)$ of the checkerboard pattern come with an orientation as well.  We define projective modules $P_0(\zg)=\oplus_i P(i)$ and $P_1(\zg)=\oplus_j P(j)$, where the first sum is over all $i$ such that $\rho(i)$ crosses $\zg$ from right to left and the second sum is over all $j$ such that $\rho(j)$ crosses $\zg$ from left to right. 

We are now ready to state our main result. 

\begin{thm}\label{thm main}
Let $B$ be a dimer tree algebra of total weight $2N$ and $\mathcal{S}$ the associated checkerboard polygon. 
For each 2-diagonal $\zg$ in $\cals$ there exists a morphism $f_\zg\colon P_1(\zg)\to P_0(\zg)$ such that the mapping $\zg\mapsto \coker f_\zg$ induces an equivalence of categories 
\[F\colon\diag \to \scmp \, B.\]
 Under this equivalence, the radical line $\rho(i)$ corresponds to the radical of the indecomposable projective $P(i)$ for all $i\in Q_0$. The clockwise rotation $R$ of $\cals$ corresponds to the shift $\zO$ in $\scmp\,B$ and $R^2$ corresponds to the  inverse Auslander-Reiten translation $\tau^{-1}=\zO^2$. 
Thus
\[ F(\rho(i))=\rad P(i)\]
\[ F(R(\zg))= \zO\,F(\zg)\]
\[ F(R^2(\zg))= \tau^{-1}\,F(\zg)\]
Furthermore, $F$ maps the 2-pivots in $\diag$ to the irreducible morphisms in $\scmp\, B$, and the  meshes in $\diag$ to the Auslander-Reiten triangles in $\scmp\,B$.
\end{thm}

An example is given in Figure~\ref{fig:dual_graph}.

Theorem~\ref{thm main} was conjectured in \cite{SS}, where it was proved in the special case where every chordless cycle in $Q$ has length 3.  Now that this conjecture is proved, we also have Corollaries 1.3-1.8 of \cite{SS} in full generality, some of which we recall now.   To begin with, the category $\scmp\,B$ is equivalent to the 2-cluster category of type $\mathbb{A}_{N-2}$ and the number of indecomposable syzygies is $N(N-2)$.  In particular, dimer tree algebras have finite Cohen-Macauley type. Furthermore, the projective resolutions of syzygies are completely determined by the checkerboard polygon $\mathcal{S}$ and are periodic of period $N$ or $2N$.  Additionally, the indecomposable syzygies are rigid, meaning they have no self-extensions in $\text{mod}\,B$.  We conjecture that they are also $\tau$-rigid.  Moreover, the total weight of $B$ is a derived invariant, even a singular invariant, for the dimer tree algebra $B$, and it can be easily read off the quiver.  Finally, the same checkerboard polygon also provides a geometric model for the stable cosyzygy category $\underline{\text{CMI}}\,B$, and we have a commutative diagram of equivalences
\[\xymatrix{
\scmp\,B \ar@<.5ex>[rr]^{\tau} && \underline{\text{CMI}}\,B \ar@<.5ex>[ll]^{\tau^{-1}}\\
&\diags\ar[ur]_{\text{ker}\,\nu f_{\zg}}\ar[ul]^{\text{coker}\,f_{\zg}}
}
\]
where $\tau, \tau^{-1}$ are the Auslander Reiten translations and $\nu$ is the Nakayama functor in $\text{mod}\,B$.

\subsubsection*{A few words about the proof}
In the proof, we establish certain derived equivalences and singular equivalences and use them to reduce the problem to the case when $B$ has only one chordless cycle, see Section~\ref{sect 1}. Some of these equivalences are given by mutation while others are given by removal or addition of vertices to the quiver. This process builds on earlier results by Bastian-Holm-Ladkani \cite{BHL}, Lu \cite{Lu}, and Chen \cite{Chen}.  In Section~\ref{sect 2}, we then give a new construction of the checkerboard pattern and show that it is equivalent to the one in \cite{SS}. The main result is proved in Section~\ref{sect 3}. 

One difference from the approach in \cite{SS} is that we show the existence of the morphism $f_{\zg}$ in the main theorem, but we do not have an explicit construction.  It would be useful to have such a description, but, judging from \cite{SS}, such a description would be quite involved.

\subsubsection*{Related work}
For the very special class of cluster-tilted algebras of finite representation type, the syzygy categories were studied before by  Chen, Geng and Lu in \cite{CGL}, where they gave a classification of the syzygy categories of these algebras. In particular, they show that the components of $\scmp \,B$ are equivalent to the stable categories of certain self-injective algebras. Their procedure involves a case by case analysis that uses a  classification of the derived  equivalence classes  of cluster-tilted algebras of Dynkin type in \cite{BHL,BHL2}. 
 Later Lu extended these results to simple polygon-tree algebras \cite{L}. One of the ingredients of the proof is successive mutation at vertices of the exterior cycles and reduction to a cluster-tilted algebra of Dynkin type $\mathbb{D}$, and we generalize this method in the first step of our proof in section~\ref{sect 1}.
 These algebras are special cases of dimer tree algebras. 
 The above results determine only the type of the syzygy category but do not describe the objects or the morphisms.

Garcia-Elsener and the first author have described the syzygy category of cluster-tilted algebras of type $\mathbb{D}$ in terms of arcs in a once-punctured polygon in \cite{GES}. 

For gentle algebras, the singularity categories have been described by Kalck in \cite{K} using $m$-cluster categories of type $\mathbb{A}_1$.  In our setting the algebra is gentle if and only if the quiver has a unique chordless cycle. This has been extended to skew-gentle algebras by Chen and Lu in \cite{CL}.
For further results on singularity categories of finite dimensional algebras see \cite{C,C2, CDZ, LZ, Sh}.

\subsubsection*{Future directions} In a work in progress, we describe the connection to dimer algebras and show how to embed our checkerboard polygon
in an alternating strand diagram of the dimer model. For an illustration, we show in Figure~\ref{fig:dual_graph2}  an alternating strand diagram that contains the checkerboard polygon of Figure~\ref{fig:dual_graph}. The orientation of the strands is such that the shaded regions are oriented, while the white regions are alternating. The corresponding dimer algebra on the disc is given by the quiver on the right in the same figure.
  Each vertex represents a white region in the alternating strand diagram and two regions are connected by an arrow if they share a crossing point.
 The full subquiver on the vertices 1, 2, \dots, 9 is equal to the twisted quiver $\widetilde Q$ of the quiver $Q$ from  Figure~\ref{fig:dual_graph} in the sense of Bocklandt \cite{Bocklandt}. The vertices 10, 11, \dots 23 are frozen vertices.  Note that both the checkerboard polygon and the alternating strand diagram have 14 boundary vertices.

 In another direction, it will be interesting to see if we can relax the conditions on the quiver such as allowing the dual graph to be disconnected or to contain cycles.   Other future projects include the behavior of the checkerboard polygon under mutations, a description of the syzygies in terms of their composition factors, and the question of $\tau$-rigidity of the indecomposable syzygies.

\subsection*{Acknowledgements:}  We thank  Eleonore Faber,  Alastair King and Matthew Pressland for interesting discussions.

The first author was supported by the NSF grants  DMS-1800860,  DMS-2054561, and by the University of Connecticut. The second author was supported by the NSF grant DMS-2054255. This work was partially
supported by a grant from the Simons Foundation.  The authors would like to thank the Isaac Newton Institute for Mathematical Sciences for support and hospitality during the programme Cluster Algebras and Representation Theory when work on this paper was undertaken. This work was supported by: EPSRC Grant Number EP/R014604/1.

\begin{figure}
\begin{minipage}{.5\textwidth}
\begin{center}
\scalebox{.7}{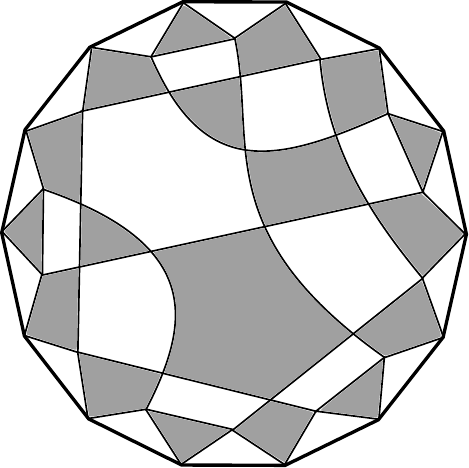}
\end{center}
\end{minipage}%
\begin{minipage}{.5\textwidth}
\[\xymatrix@R=6pt@C=6pt{
&&{\color{blue} 23}\ar[dr] && {\color{blue} 10}\ar[ll]\ar[dr]\\
&{\color{blue} 22}\ar[dl]\ar[ur]&&1\ar[dd]\ar[ur]&&{\color{blue} 11} \ar[dl]\\
{\color{blue} 21}\ar[dr]&&&&3\ar[dr]\ar[ul] && {\color{blue} 12} \ar[ul]\ar[dd]\\
&9\ar[rr]\ar[dl]&&2\ar[dl]\ar[ur] \ar[uull]&& 5\ar[dl]\ar[ur]\\
{\color{blue} 20}\ar[d]\ar[uu]&&8\ar[d]\ar[ul]&&4\ar[ul] \ar[dr]&& {\color{blue} 13} \ar[ul]\\
{\color{blue} 19}\ar[urr]&&7\ar[r] \ar[dl]&6\ar[ur] \ar[dd]&& {\color{blue} 14} \ar[ur] \ar[dl]\\
&{\color{blue} 18}\ar[dr]\ar[ul]&&&{\color{blue} 15} \ar[ul]\\
&&{\color{blue} 17}\ar[uu]&{\color{blue} 16} \ar[ur]\ar[l]
}
\]
\end{minipage}
\caption{An alternating strand diagram  that contains the checkerboard polygon of Figure~\ref{fig:dual_graph}. The corresponding quiver is shown on the right. Its mutable part is the twist of the quiver of the checkerboard polygon.}
\label{fig:dual_graph2}
\end{figure}

\section{Preliminaries} 
Let $\kb$ be an algebraically closed field. If $A$ is a finite-dimensional $\kb$-algebra, we denote by $\textup{mod}\,A$ the category of finitely generated right $A$-modules. Let $D$ denote the standard duality $D=\Hom(-,\kb)$. If $Q_A$ is the ordinary quiver of the algebra $A$, and $i$ is a vertex of $Q_A$,  we denote by $P(i),I(i),S(i)$ the corresponding indecomposable projective, injective, simple $A$-module, respectively. 

Let $\rad\, A$ denote the Jacobson radical of $A$. If $M\in \textup{mod}\,A$ its radical is defined as $\rad \,M= M (\rad\,A)$ and its top
as $\textup{top}\,M= M/\rad\, M$. Thus in particular $\textup{top}\,P(i) =S(i)$.
Given a module $M$, we denote by $\add \,M$ 
the full subcategory of $\textup{mod}\,A$ whose objects are direct sums of summands of $M$.

For further information about representation theory and quivers we refer to \cite{ASS,S2}.

\subsection{Cohen-Macauley modules over 2-Calabi-Yau tilted algebras}\label{sect CM}
From now on, let $B$ be a 2-Calabi-Yau tilted algebra. Thus $B$ is the endomorphism algebra of a cluster-tilting object in a 2-Calabi-Yau category. 
A $B$-module $M$ is said to be projectively Cohen-Macauley if $\Ext^i_B(M,B)=0$ for all $i>0$. In other words, $M$ has no extensions with projective modules.

We denote by $\cmp\,B$ the full subcategory of $\textup{mod}\,B$ whose objects are the projectively Cohen-Macauley modules. This  is a Frobenius category. The projective-injective objects in $\cmp\,B$ are are precisely the projective $B$-modules. The corresponding stable category $\scmp\,B$ is triangulated, and its inverse shift is given  by the syzygy operator $\zO$ in $\textup{mod}\,B$.

Moreover, by Buchweitz's theorem \cite[Theorem 4.4.1]{Bu}, there exists a triangle equivalence between $\scmp\,B$ and the singularity category $\cald^b(B)/\cald^b_{perf} (B)$ of $B$.
Keller and Reiten showed in \cite{KR} that the category $\scmp \,B$ is 3-Calabi-Yau.

It was shown in  \cite{GES} that if $M\in \text{mod}\,B$ is indecomposable then the following are equivalent. 
 \begin{itemize}
\item [(a)] $M$ is a non-projective syzygy;
\item [(b)] $M \in \textup{ind}\,\scmp\,B$; 
\item [(c)] $\zO^2_B \tau_B M \cong M$.
\end{itemize}

We may therefore use the terminology ``syzygy'' and ``Cohen-Macauley module'' interchangeably.

Two algebras are said to be \emph{derived equivalent} if there exists a triangle equivalence between their bounded derived categories.
Two algebras are said to be \emph{singular equivalent} if there exists a triangle equivalence between their singularity categories.

\subsection{Quivers with potentials}\label{sect QP}
A quiver $Q=(Q_0,Q_1,s,t)$ consists of a finite set of vertices $Q$, a finite set of arrows $Q_1$ and two maps $s,t\colon Q_1\to Q_0$, where $s$ is the source and $t$ is the target of the arrow. Thus if $\za\in Q_1$ then $\za\colon s(\za)\to t(\za)$. 

 A \emph{potential} $W$ on a quiver $Q$ is a  linear combination of non-constant  cyclic paths. For every arrow $\za\in Q_1$, the cyclic derivative $\partial _\za$ is defined on a cyclic path $\za_1\za_2\dots\za_t$ as 
\[ \partial_\za(\za_1\za_2\dots\za_t)=\sum_{p\colon \za_p=\za} \za_{p+1}\dots\za_t\za_1\dots\za_{p-1}\]
and extended linearly to the potential $W$.

The \emph{Jacobian algebra} $\textup{Jac}(Q,W)$ of the quiver with potential is the quotient of the (completed) path algebra $\kb Q$ by  (the closure of) the 2-sided ideal generated by all partial derivatives $\partial_\za W$, with $\za\in Q_1$. 
Two parallel paths in the quiver are called \emph{equivalent} if they are equal in $\textup{Jac}(Q,W)$.

If $Q$ has no oriented 2-cycles then $\textup{Jac}(Q,W)$ is 2-Calabi-Yau tilted by \cite{Amiot}.

Let $(Q,W)$ be a quiver with potential and $k$ a vertex of $Q$ that does not belong to an oriented 2-cycle. Using a cyclic shift if necessary, we may assume without loss of generality that none of the cyclic paths in $W$ starts at $k$. 
The mutation of $(Q,W)$ at $k$ is the quiver with potential $\mu_k(Q,W)=(Q',W')$, where

(i) $Q'$ is the quiver obtained from $Q$ by the following local transformations
\begin{enumerate}
\item[-] For every path of length two $\xymatrix@C15pt{i\ar[r]^\za&k\ar[r]^\zb&j} $,  introduce a new arrow  $\xymatrix{i\ar[r]^{[\za\zb]} &j}$. 
\item[-] Replace every arrow $\za\colon i\to k$ ending at $k$ by its opposite $\zabar\colon i\ot k$
and every arrow $\zb\colon k\to j$ starting at $k$ by its opposite  $\zbbar\colon k\ot j.$
\end{enumerate}

(ii) $W'=[W]+\zD_k$, where 
\[\zD_k=\sum_{\za,\zb\in Q_1\colon t(\za)=s(\zb)=k} [\za\zb]\,\zbbar\,\zabar\]
and 
$[W]$ is obtained from $W$ by substituting $[\za\zb]$ for each factor $\za\zb$ with $t(\za)=s(\zb)=k$.

A potential that does not contain any 2-cycles is called \emph{reduced}. It may happen that the mutation of a reduced potential is no longer reduced. However, often it is possible to replace $W'$ by an equivalent potential that is reduced. In that case, we also may remove all 2-cycles from the quiver, and then we obtain the quiver $Q'=\mu_k Q$ given by ordinary quiver mutation at $k$.

\subsection{Translation quivers and mesh categories} We review here the notions of translation quiver and mesh category from \cite{Ri, H}. These notions are often used in order to define a category from combinatorial data. Examples of such constructions are the combinatorial constructions of cluster categories of finite type in \cite{BM, BM2, CCS, S}. 

A {\em  translation
 quiver} $(\zG,\tau)$ is a quiver $\zG=(\zG_0,\zG_1)$ without loops
 together with an injective map $\tau\colon \zG_0'\to\zG_0$  (the {\em translation}) from a subset $\zG_0'$ of $\zG_0$ to $\zG_0$ such that, for all vertices $x\in\zG_0'$, $y\in \zG_0$, 
the number of arrows from $y \to x$ is equal to the number of arrows
 from $\tau x\to y$. 
Given a  translation quiver $(\zG,\tau)$, a \emph{polarization of} $\zG$ is
 an injective map $\sigma:\zG_1'\to\zG_1$, where $\zG_1'$ is the set of all arrows $\za\colon y\to x$ 
 with $x \in \zG_0'$, such that 
$\sigma(\za)\colon \tau x\to y$  for every arrow $\za\colon y\to x\in \zG_1$.
From now on we assume that $\zG$ has no multiple arrows. In that case, there is a unique polarization of $\zG$.

The {\em path category } of a translation quiver $(\zG,\tau)$ is the category whose  objects are
the vertices $\zG_0$ of $\zG$, and, given $x,y\in\zG_0$, the $\kb$-vector space of
morphisms from $x$ to $y$ is given by the $\kb$-vector space with basis
the set of all paths from $x $ to $y$. The composition of morphisms is
induced from the usual composition of  paths.
The {\em mesh ideal} in the path category of $\zG$ is the ideal
generated by the {\em mesh relations}
\begin{equation}\nonumber
m_x =\sum_{\za:y\to x} \sigma(\za) \za 
\end{equation}  
for all $x \in \zG_0'$.

The {\em  mesh category } of the translation quiver $(\zG,\tau)$ is the
quotient of its path 
category by the mesh ideal.

\subsection{The category of 2-diagonals of a polygon} \label{sect 23} In this subsection, we recall a geometric model for the 2-cluster categories of type $\mathbb{A}$ obtained by Baur and Marsh in \cite{BM}.

 Let $S$ be a regular polygon with  an even number of vertices, say $2N$. Let $R$ be the automorphism of $S$  given by a clockwise rotation about $180/N $ degrees. Thus $R^{2N}$ is the identity. 

Following \cite{BM}, we define the category $\textup{Diag}(S)$ of 2-diagonals of $S$ as follows. 
The indecomposable objects of $\textup{Diag}(S)$ are the 2-diagonals in $S$. Recall that a 2-diagonal is a diagonal of $S$ connecting two vertices such that the two polygons obtained by cutting $S$ along the diagonal both have an even number of vertices and at least 4. In particular, boundary segments are not 2-diagonals.

The irreducible morphisms of $S$ are given by 2-pivots. We recall the definition below. An illustration is given in Figure \ref{fig:S} on the left. 
\begin{definition}\label{def 2pivot}
 Let $\zg$ be a 2-diagonal in the polygon $S$ and denote its endpoints by $a$ and $x$. Denote by $b$ the clockwise neighbor of $a$, and by $c$ the clockwise neighbor of $b$ on the the boundary of $S$. At the other end, denote by $y$ the clockwise neighbor of $x$, and by $z$ the clockwise neighbor of $y$ on the boundary of $S$. 
 
Unless $a$ and $z$ are neighbors on the boundary, the 2-diagonal $\zg'$  connecting $a$ and $z$ is called the {\em 2-pivot of $\zg$ fixing the endpoint $a$}.

Unless $c$ and $x$ are neighbors on the boundary, the 2-diagonal $\zg''$  connecting $c$ and $x$  is called the {\em 2-pivot of $\zg$ fixing the endpoint $x$}.
\end{definition}
\begin{figure}
\begin{center}
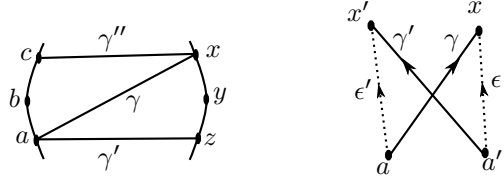
\caption{ The arc $\zg$ and its 2-pivots $\zg', \zg''$ are shown on the left, and the crossing between $\zg$ and $\zg'$ on the right gives rise to an extension in $\text{Ext}^1(M_{\zg}, M_{\zg'})$.}
\label{fig:S}
\end{center}
\end{figure}

Let $\zG$ be the quiver whose vertices are the 2-diagonals in $S$, and there is an arrow from the 2-diagonal $\zg$ to the 2-diagonal $\zg'$ precisely if $\zg'$ is obtained from $\zg$ by a 2-pivot. Then the  pair  $(\zG, R^{-2})$ is a translation quiver.

\begin{definition}\cite{BM}
 The category $\textup{Diag}(S)$    of 2-diagonals in the  polygon $S$ is the mesh category of the translation quiver $(\zG,R^{-2})$.
\end{definition}

Next, we recall that the category of diagonals is equivalent to the 2-cluster category.  
Let  $H$ be the path algebra of a Dynkin quiver of type $\mathbb{A}_r$. Let $\calc^2$ denote the 2-cluster category of type $\mathbb{A}_{r}$. This category is defined as the orbit category of the bounded derived category $\cald^b(\textup{mod}\,H)$ by the functor $\tau_{\cald}^{-1}[2]$. 
Here $\tau_{\cald}$ is the Auslander-Reiten translation and $[2]=[1]\circ [1]$ is the second power of the shift functor in the derived category.  Thus
\begin{equation*}
\calc^2=\cald^b(\textup{mod}\,H)/\tau_{\cald}^{-1}[2].\end{equation*}
This category was introduced in \cite{K,T}, and was studied in \cite{BRT,IY, Tor}.

\begin{thm}\label{thm BM}
 \cite{BM} Let $S$ be a polygon with $2N$ vertices. Then the category $\textup{Diag}(S)$ is equivalent to the 2-cluster category of type $\mathbb{A}_{N-2}$.
\end{thm}

Under this equivalence each 2-diagonal $\zg$ of $S$ corresponds to an indecomposable object $M_{\zg}$ in $\calc^2$. Moreover, there exists a nontrivial extension between two indecomposable objects in $\calc^2$ if and only if
the corresponding 2-diagonals cross. 

Next, we introduce orientations on the diagonals to determine the direction of the extension between the two objects in $\calc^2$ whose corresponding 2-diagonals cross.  Label the boundary vertices of $S$ from 1 to $2N$ in clockwise order.  Then the two endpoints of any 2-diagonal have different parity, and we orient a 2-diagonal from its odd labeled endpoint to its even labeled endpoint.    Given 2-diagonals $\zg, \zg'$ that cross, we say that {\it $\zg'$ crosses $\zg$ from right to left} (respectively left to right) in the situations shown in Figure~\ref{fig crossdirection}. Using this notation we obtain the following result.  
\begin{figure}
\begin{center}
\scalebox{0.7}{\Large
\begingroup%
  \makeatletter%
  \providecommand\color[2][]{%
    \errmessage{(Inkscape) Color is used for the text in Inkscape, but the package 'color.sty' is not loaded}%
    \renewcommand\color[2][]{}%
  }%
  \providecommand\transparent[1]{%
    \errmessage{(Inkscape) Transparency is used (non-zero) for the text in Inkscape, but the package 'transparent.sty' is not loaded}%
    \renewcommand\transparent[1]{}%
  }%
  \providecommand\rotatebox[2]{#2}%
  \newcommand*\fsize{\dimexpr\f@size pt\relax}%
  \newcommand*\lineheight[1]{\fontsize{\fsize}{#1\fsize}\selectfont}%
  \ifx\svgwidth\undefined%
    \setlength{\unitlength}{295.41875859bp}%
    \ifx\svgscale\undefined%
      \relax%
    \else%
      \setlength{\unitlength}{\unitlength * \real{\svgscale}}%
    \fi%
  \else%
    \setlength{\unitlength}{\svgwidth}%
  \fi%
  \global\let\svgwidth\undefined%
  \global\let\svgscale\undefined%
  \makeatother%
  \begin{picture}(1,0.26119572)%
    \lineheight{1}%
    \setlength\tabcolsep{0pt}%
    \put(0,0){\includegraphics[width=\unitlength,page=1]{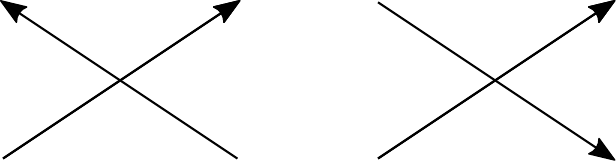}}%
    \put(0.07602553,0.07977283){\makebox(0,0)[lt]{\lineheight{1.25}\smash{\begin{tabular}[t]{l}$\zg$\end{tabular}}}}%
    \put(0.2842046,0.07977283){\makebox(0,0)[lt]{\lineheight{1.25}\smash{\begin{tabular}[t]{l}$\zg'$\end{tabular}}}}%
    \put(0.68533009,0.07977283){\makebox(0,0)[lt]{\lineheight{1.25}\smash{\begin{tabular}[t]{l}$\zg$\end{tabular}}}}%
    \put(0.89350891,0.07977283){\makebox(0,0)[lt]{\lineheight{1.25}\smash{\begin{tabular}[t]{l}$\zg'$\end{tabular}}}}%
  \end{picture}%
\endgroup%
}
\caption{Crossings of oriented 2-diagonal. In the left picture $\zg'$ crosses $\zg$ from right to left, and in the right picture $\zg'$ crosses $\zg$ from left to right.}
\label{fig crossdirection}
\end{center}
\end{figure}
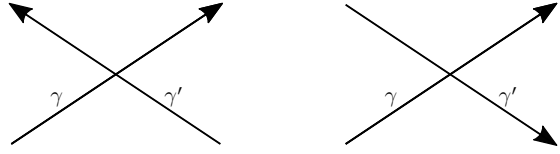

\begin{lemma}\label{lem left to right}
Let $\zg, \zg'$ be oriented 2-diagonals in $\textup{Diag}(S)$ that cross.  
\begin{itemize}
\item[(a)] $\zg'$ crosses $\zg$ from right to left if and only if $\textup{Ext}^1(M_\zg, M_{\zg'})\not=0$.
\item[(b)] $\zg'$ crosses $\zg$ from left to right if and only if $\textup{Ext}^1(M_{\zg'}, M_\zg)\not=0$.
\end{itemize}
\end{lemma}

\begin{proof}
Let $\zg$ have endpoints $a,x$ and $\zg'$ have endpoints $a',x'$ as in Figure~\ref{fig:S} on the right.   

By Theorem~\ref{thm BM}, the crossing of $\zg, \zg'$ corresponds to an extension between $M_\zg, M_{\zg'}$ in the 2-cluster category. Moreover, $\textup{Ext}^1(M_\zg, M_{\zg'})\not=0$ if and only if there is a triangle in the 2-cluster category as follows, where at most one of $M_\ze,M_{\ze'}$ may be zero.  
\begin{equation}\label{eqS}
\xymatrix{
M_{\zg'}\ar[r]^-{\begin{bsmallmatrix}f\\f'\end{bsmallmatrix}} & M_{\epsilon}\oplus M_{\epsilon'} \ar[r]^-{\begin{bsmallmatrix}g & g'\end{bsmallmatrix}} & M_\zg \ar[r]& M_{\zg'}[1]}
\end{equation}
Here each $f, f'$ is a sequence of irreducible morphisms corresponding to a sequence of 2-pivots such that each 2-pivot in $f$  fixes the endpoint $a'$ of $\zg'$  and takes $\zg'$ to $\epsilon$ and each 2-pivot in $f'$  fixes the endpoint $ x'$ of $\zg'$  and takes $\zg'$ to $\epsilon'$.  Similarly, $g,g'$ each correspond to a sequence of 2-pivots such that each 2-pivot in $g$ fixes the endpoint $x$ of $\epsilon$ and takes $\epsilon$ to $\zg$, and  each 2-pivot in $g'$ fixes the endpoint $a$ of $\epsilon'$ and takes $\epsilon'$ to $\zg$.  In particular, $\epsilon$ and $ \epsilon'$ are 2-diagonals.

Without loss of generality we may assume that $\epsilon$ is oriented from $a'$ to $x$, since these are 2-diagonals.  This means that $\zg'$ is oriented from $a'$ to $x'$ and $\zg$ is oriented from $a$ to $x$, see Figure~\ref{fig:S} on the right.  Hence, we see that $\zg'$ crosses $\zg$ from right to left if $\textup{Ext}^1(M_\zg, M_{\zg'})\not=0$.  Note that if instead we were to suppose that $\epsilon$ is oriented from $x$ to $a'$, we would still conclude that $\zg'$ crosses $\zg$ from right to left, because both $\zg, \zg'$ would change orientation.  Conversely, if $\zg'$ crosses $\zg$ from right to left, then we can construct the triangle in \eqref{eqS} with the desired properties.  Note that the line segments with endpoints $x,x'$ and $a,a'$ respectively are not 2-diagonals, because of the orientation on $\zg$ and $\zg'$, and hence $\epsilon, \epsilon'$ correspond to the middle terms of the triangle.  
This shows part (a), and part (b) follows directly from (a).
\end{proof}


\section{Dimer tree algebras}
In this section, we define a  class of Jacobian algebras that are the subject of this paper. An example is given in the introduction.

 A \emph{chordless cycle} in a quiver $Q$ is a cyclic path $C=x_0\to x_1 \to\dots\to x_t\to x_0$ such that $x_i\ne x_j$ if $i\ne j$ and the full subquiver on vertices $x_0,x_1,\dots, x_t$ is equal to $C$. 
The arrows that lie in exactly one chordless cycle will be called {\em boundary arrows} and those that lie in two or more chordless cycles {\em interior arrows} of $Q$. 

\begin{definition}
  The \emph{dual graph} $G$ of $Q$ is defined as follows. The set of vertices $G_0$ is the union of 
the set of chordless cycles of $Q$ and the set of boundary arrows of $Q$. 
The set of edges $G_1$ is the union of two sets
called the set of \emph{trunk edges} and the set of \emph{leaf branches}.  A trunk edge $\xymatrix{C\ar@{-}[r]^\za&C'}$ is drawn between any pair of chordless cycles  $(C,C')$ that share an arrow $\za$. A leaf branch $\xymatrix{C\ar@{-}[r]^\za &\za}$ is drawn between any pair $(C,\za)$ where $C$ is a chordless cycle and $\za$ is a boundary arrow such that $\za $ is contained in $C$.
\end{definition}

\begin{definition}[The quiver]
\label{def Q}  Throughout the paper,  we let $Q$ be a finite connected quiver without loops and 2-cycles satisfying the following conditions. 
\begin{itemize}
\item [(Q1)] Every arrow of $Q$ lies in at least one chordless cycle.
\item[(Q2)]\label{tree} The dual graph of $Q$ is a tree. 
\item[(Q3)] The boundary arrows of $Q$ form a simple (non-oriented) cycle.
\end{itemize}
\end{definition}

The following properties follow easily from the definition.
\begin{prop}\label{prop Q}\cite[Proposition 3.4]{SS}  Let $Q$ be a quiver satisfying Definition~\ref{def Q}.
 \begin{enumerate}
\item $Q$ has no parallel arrows.
\item $Q$ is planar.
\item For all arrows $\za $ of $Q$,
 \begin{enumerate}
\item[(a)] either $\za$ lies in exactly one chordless cycle, 
\item[(b)] or $\za$ lies in exactly two chordless cycles.\end{enumerate}
\item Any two chordless cycles in $Q$ share at most one arrow.
\end{enumerate}
\end{prop}

Since the dual graph $G$ of $Q$  is a tree, the quiver contains a chordless cycle $C_0$ that contains exactly one interior arrow. If $C$ is any chordless cycle in $Q$ we define the distance $d(C)$ of $C$ from $C_0$ to be  the length of the unique path from $C_0$ to $C$ in $G$.

\begin{definition}\label{def B} 
Let $Q$ be a quiver that satisfies Definition~\ref{def Q} and let $W=\sum_{C} (-1)^{d(C)} C$ be its potential, where the sum is taken over all chordless cycles of $Q$. The the Jacobian algebra $B=\textup{Jac(Q,W)}$ is called a \emph{dimer tree algebra}.
\end{definition}
In other words, the chordless cycle $C$ appears with a positive sign in the potential if and only if the path from $C_0$ to $C$ is of even length.

\begin{remark}
 \label{rem dimer tree}
 A dimer tree algebra is not strictly speaking a dimer algebra, since the boundary arrows in a dimer tree algebra also induce relations. Still there are many similarities, in particular, the zigzag paths in a dimer algebra correspond to our cycle paths defined below.
\end{remark}

An algebra is called \emph{schurian} if $\dim\Hom(P(i),P(j))\le 1$ for all vertices $i,j$ of $Q$.
\begin{prop}\cite[Corollary 3.31]{SS}
 \label{prop schurian}
 Every dimer tree algebra is a schurian algebra. In particular, any two nonzero parallel paths are equal, and any non-constant cyclic path is zero.
\end{prop}

\subsection{Cycle paths, weight, and total weight}
We review some definitions from \cite{SS} relating to the quiver $Q$.  Recall that an arrow in $Q$ is called a boundary arrow if it lies in exactly one chordless cycle.

\begin{definition}
 \label{def cycle path 2}
 Let $\za$ be a  boundary arrow in $Q$.  
 \begin{itemize}
\item[(a)] The 
 \emph{cycle path} of $\za$ is the unique path 
 $\mathfrak{c}(\za)=\za_1\za_2\cdots\za_{\ell(\za)}$ such that 
 \begin{itemize}
\item [(i)] $\za_1=\za$ and $\za_{\ell(\za)}$  are boundary arrows, and 
 $\za_2,\ldots,\za_{\ell(\za)-1}$ are interior arrows,
\item [(ii)] every subpath of length two $\za_i\za_{i+1}$, 
 is a subpath of a chordless cycle $C_i$, and $C_i\ne C_j$ if $i\ne j$. 
\end{itemize}

\item[(b)] The \emph{weight} $\text{w}(\za)$ of $\za$ is defined as
\[\textup{w}(\za) =\left\{ 
\begin{array}
 {ll} 1&\textup{if the length of  $\mathfrak{c}(\za)$ is odd;}\\
 2&\textup{if the length of  $\mathfrak{c}(\za)$ is even.}\\
\end{array} \right.\]

\item[(c)] Dually, the path $\mathfrak{c}(\za)$ is uniquely determined by the last arrow $\za_{\ell(\za)}$, and it is called the \emph{cocycle path} of $\za_{\ell(\za)}$.  Define the \emph{coweight} $\overline{\text{w}}(\za_{\ell(\za)})$ of $\za_{\ell(\za)}$ to be equal to $\text{w}(\za)$.

\end{itemize}
\end{definition}

Define 
\[\sum_\za \textup{w}(\za) = \sum_\za \overline{\textup{w}}(\za),\]
where the sum is over all boundary arrows of $Q$, to be the \emph{total weight} of $Q$.  The total weight is an important statistic of the quiver because of the following result.

\begin{prop}\cite[Corollary 3.17]{SS}
 \label{cor size of S}
The number of boundary edges in the checkerboard polygon $\cals$ is equal to the total weight of $Q$. 
\end{prop}

In the example of Figure~\ref{fig:dual_graph}, we list the cycle paths and their corresponding weights below.   In particular, we see that the total weight equals 14, which is the same as the size of the corresponding checkerboard polygon $\mathcal{S}$. 

\[
\begin{matrix}
\underline{\text{cycle path}} && \underline{\text{weight}}\\
1\to 2\to 3\to 4\to 6\to 9 && 1\\
3\to 1\to 2 && 2\\
8\to 3\to 4 \to 5 && 1\\
7\to 8 \to 3 && 2\\
6\to 7 \to 8 && 2\\
6\to 9 \to 4 && 2\\
9\to 4 \to 6 \to 7 && 1\\
4\to 5 \to 2 && 2\\
5\to 2\to 3\to 1 && 1
\end{matrix}
\]

\begin{corollary}
The total weight of a dimer tree algebra is a derived invariant and a singular invariant of the algebra. 
\end{corollary}

\begin{remark}
In the dimer terminology, a cycle path is usually called a zigzag path.  In the checkerboard polygon the cycle paths are realized by going clockwise around the white regions. 
\end{remark}

\section{A triangle equivalence $\scmp \,B\cong \diags$}\label{sect 1}
In this section, we generalize a result of \cite{Lu} which will show that the category $\scmp \,B$ of non-projective syzygies over $B$ is triangle equivalent to the category $\diags$ of 2-diagonals of a polygon $S$, see Theorem~\ref{thm algo}.

\subsection{Recollections} We start by recalling several results from the literature.
The following result was proved by Chen.
\begin{thm}
 \label{thm chen} \cite[Proposition 3.1]{Chen}
 Let $A$ be a finite dimensional $\kb$-algebra, ${}_AM$ a left $A$-module and $N_A$ a right $A$-module. Let $\phi\colon M\otimes_{\kb} N\to A$ be a monomorphism. Then $\im \,\phi$ is a two-sided ideal in $A$. Assume further that $(\im\, \phi) M=0$ and $N(\im\, \phi)=0$. Define the matrix algebra
 \[\zG=\begin{pmatrix} A&M\\N&\kb\end{pmatrix},\]
 whose multiplication is given by the formula
 \[ 
\begin{pmatrix}
 a&m\\n&\zl 
\end{pmatrix}
\begin{pmatrix}
 a'&m'\\n'&\zl' 
\end{pmatrix}
=
\begin{pmatrix}
 aa'+\phi(m\otimes n')&am'+\zl'm\\
 na'+\zl n'&\zl\zl' 
\end{pmatrix}.
\]
 Then there is a triangle equivalence between the singularity categories of $\zG$ and $A/\im \,\phi$. 
\end{thm}

We will need the following  corollary about the special case of one-point extensions and co-extensions.
With the notation above, we say that \begin{enumerate}
\item $\zG$ is a \emph{one-point extension} of $A$ if $M=0$ and $N$ is a projective module.
\item $\zG$ is a \emph{one-point co-extension} of $A$ if $N=0$ and (the dual of) $M$ is an injective module. 
\end{enumerate}
In both cases the tensor product $M\otimes_\kb N$ is trivial and $\phi$ is the zero morphism.
\begin{corollary}
 \label{cor Chen 1}
 If $\zG$ is a one-point extension or a one-point co-extension of $A$ then $\zG$ and $A$ are singular equivalent.
\end{corollary}
\begin{remark}
If $M=\oplus P(i)$ then  the quiver of the one-point extension $\zG$ is obtained from the quiver of $A$ by adding one vertex $x$ and adding one arrow $x\to i$ for each indecomposable summand $P(i)$ of $M$. Moreover, none of the new arrows appears in a relation in $\zG$.
\end{remark}

The following result was obtained in \cite{L,BHL, Lu} for arbitrary schurian algebras. We reformulate it here in our setting. If $B=\textup{Jac}(Q,W)$ be a dimer tree algebra and $k$ is a vertex of $Q$, we define the mutation at $k$ as the Jacobian algebra of the quiver with potential obtained by mutation at $k$, see section~\ref{sect QP}. In other words
$\mu_k B=\textup{Jac} (\mu_k(Q,W))$. 
\begin{thm} \label{thm BHL}
\cite[Propositions 2.9 \& 2.15]{BHL}\cite[Proposition 2.17]{Lu}
With the above notation, the algebras $B$ and $\mu_k B$ are derived equivalent if 
\begin{enumerate}
\item [(i)] before the mutation, there is at most one arrow ending in $k$ and, if $v$ is a nonzero path in $B$ ending at $k$, then there exists an arrow $\za$ such that $v\za\ne 0$ in $B$.
\item [(ii)] after the mutation, there is at most one arrow starting at $k$ and, if $v$ is a nonzero path starting at $k$ in $\mu_k B$, then there exists an arrow $\za$ such that $\za v\ne 0$ in $\mu_k B$.
\end{enumerate}

\end{thm}

%
%
\subsection{Extending nonzero paths}
Let $B$ be a dimer tree algebra. 
In this subsection, we use the weight of a boundary arrow to characterize when a nonzero path in $B$ can be extended by a boundary arrow.

A \emph{prefix} of a path $v$ in $Q$ is a subpath $u$ such that $v=uv'$ in $Q$. 
A \emph{suffix} of a path $v$ in $Q$ is a subpath $u$ such that $v=v'u$ in $Q$. 

\begin{lemma}
 \label{lem 1}
 Let $\zg\colon i\to j$ be a boundary arrow of $Q$. Then the following are equivalent.
 \begin{enumerate}
\item For all nonzero paths $v$ ending at $i$ the composition $v\zg$ is nonzero.
\item The weight $\wt(\zg)$ of $\zg$ is 1.
\end{enumerate}
\end{lemma}
\begin{proof}
Let $\mathfrak{c}(\zg)=\zg\zd_0\zd_1\dots\zd_t$ be the cycle path of $\zg$. Thus $\wt(\zg)=1$ if $t$ is odd, and $\wt(\zg)=2$ if $t$ is even.
  Let $C_0,C_1,\dots ,C_t$ be the chordless cycles along $\mathfrak{c}(\zg)$ in order and define the path $u_i$ by $C_0=\zg\zd_0u_0$ and $C_i=\zd_{i-1}\zd_iu_i$ for $i=1,2,\ldots,t$, see Figure~\ref{fig lem 1}.   Note that $\zd_t$ is a boundary arrow and all $\zd_i$ with $i<t$ are interior arrows.
    \begin{figure}
\begin{center}
\[\xymatrix{
i\ar[d]_\zg \\
j \ar[r]^{\zd_0} &\cdot\ar[d]_{\zd_1} \ar@/_16pt/@{.>}[ul]_{u_0} \\
&\cdot \ar[r]^{\zd_2} \ar@/^16pt/@{.>}[ul]^{u_1}  &\cdot \ar@/_16pt/@{.>}[ul]_{u_2}\ar[d]_{\zd_3} \\
&&\cdot  \ar@/^16pt/@{.>}[ul]^{u_3} &\dots
}
\]
\caption{Proof of Lemma \ref{lem 1}. The cycle path $\mathfrak{c}(\zg)=\zg\zd_0\zd_1\dots\zd_t$ and the chordless cycles $C_i=\zd_{i-1}\zd_iu_i$.}
\label{fig lem 1}
\end{center}
\end{figure}

  (1) $\Rightarrow$ (2).  Suppose $\wt(\zg)=2$, so $t$ is even. Then the path $v=u_tu_{t-2}\dots u_2u_0$ is nonzero ending at $i$ and, using the relations $\partial_{\zd_{i}}W=\pm(u_i\zd_{i-1} - \zd_{i+1}u_{i+1})$, we get  
  \[
\begin{array}{ccccccccccccc}
v\zg =   u_tu_{t-2}\dots u_2u_0\zg 
   =     u_tu_{t-2}\dots u_2\zd_1 u_1
     =   u_tu_{t-2}\dots \zd_3u_3 u_1
    =   u_t\zd_{t-1} u_{t-1}\dots u_3 u_1\\
\end{array}
\]
which is zero, because $u_t\zd_{t-1}=\partial _{\zd_t}W=0$.

 (2) $\Rightarrow$ (1). 
 Suppose  $v$ is a nonzero path ending at $i$ such that $v\zg=0$. We want to show that $t$ is even. Since $v\zg=0$, there must be a relation that involves both a path equivalent to $v$ and the arrow $\zg$, and that relation must end with the arrow $\zg$. Since $\zg$ is a boundary arrow, the only relation that ends in $\zg$ is
\begin{equation}
\label{eq relation}
{\partial_{\zd_0} W=\left\{
\begin{array}
 {ll}
 u_0\zg &\textup{if $t=0$;}\\
\pm( u_0\zg -\zd_1u_1)&\textup{if $t>0$.}\qquad
\end{array}\right.
}
\end{equation}

If $t=0$ there is nothing to show.  Assume $t>0$. Then without loss of generality we may assume that the path $u_0$ is a suffix of $v$. Thus there is a path $v_1$ such that 
\[v=v_1 u_0.\]
Using the relation (\ref{eq relation}) we get
\begin{equation}
\label{eq lem1 1}
0= v\zg=v_1u_0\zg=v_1\zd_1u_1.
\end{equation}
Now, $v_1$ is a nonzero path, since it is a subpath of $v$. Furthermore, $v_1$ does not end with the arrow $\zd_0$, since $\zd_0u_0=0$. 
Thus equation~(\ref{eq lem1 1}) implies that there must be a relation involving a path equivalent to $v_1$ and the arrow $\zd_1$.

If $t=1$ then $\zd_1$ is a boundary arrow and  hence the only relations involving $\zd_1$ come from the chordless cycle $C_1=\zd_0\zd_1u_1$. Since $v$ cannot end in $\zd_0$, we see that $v\zg\ne 0$, a contradiction.  Now suppose $t>1$. This means that $\zd_1$ is an interior arrow and thus there exists another 
chordless cycle $C_2=\zd_1\zd_2u_2$ containing $\zd_1$.
Thus the relation in question must be
 \begin{equation}
\label{eq relation 2}
\partial_{\zd_2} W=\left\{
\begin{array}
 {ll}
 u_2\zd_1 &\textup{if $t=2$;}\\
 \pm(u_2\zd_1 -\zd_3u_3)&\textup{if $t>2$.}
\end{array}\right.
\end{equation}

If $t=2$ there is nothing to show, so we may assume $t>2$. Then, without loss of generality, the path $u_2$ is a suffix of $v_1$, thus $v_1=v_2u_2$ and
\begin{equation}
\label{eq lem 1 2}
v\zg=v_1\zd_1 u_1 =v_2 u_2\zd_1 u_1 = v_2\zd_3 u_3 u_1,
\end{equation}
where the last equality uses relation (\ref{eq relation 2}).
Again, $v_2$ is a nonzero path and it does not end with the arrow $\zd_2$, because otherwise, the path $v$  would contain the zero subpath $\zd_2 u_2 u_0 = u_1\zd_0u_0$.
If $t=3$ then $\zd_3 $ is a boundary arrow and  hence the only relations involving $\zd_3$ come from the chordless cycle $C_3=\zd_2\zd_3u_3$. Since $v_2$ cannot end in $\zd_2$, we see that $v\zg\ne 0$, a contradiction.   So we may assume $t>3.$

Continuing this way, we see that whenever $t=2s+1$ is odd then  $v\zg\ne 0$ and we obtain a contradiction. This process must stop, because the cycle path is a well-defined finite path. 
\end{proof}

We also have the dual statement of Lemma \ref{lem 1}.
\begin{lemma}
 \label{lem 1 dual}
 Let $\zg\colon i\to j$ be a boundary arrow of $Q$. Then the following are equivalent.
 \begin{enumerate}
\item For all nonzero paths $v$ starting at $j$ the composition $\zg v$ is nonzero.
\item The coweight $\cwt(\zg)$ of $\zg$ is 1.
\end{enumerate}
\end{lemma}

\subsection{Derived equivalences given by mutation}
Let $B=\textup{Jac}(Q,W)$ be a dimer tree algebra. In this subsection, we describe two specific situations where the mutation of the quiver with potential at a vertex $k$ yields a derived equivalent algebra.
\begin{lemma}
 \label{lem der} 
If $Q$ contains one of the following two subquivers, then the mutation at the vertex $k$ is a derived equivalence that preserves the total weight of the quivers. Moreover, the mutated quiver with potential again defines a dimer tree algebra.

\textup{(a)} \[ \xymatrix{&k\ar[rd]^\zb\\
\cdot\ar[ru]^\za&&\cdot\ar@/_0pt/@{.>} [ll]_u}\] where $\za,\zb$ are boundary arrows with $\cwt(\za)=1$ and $\wt(\zb)=2$ and $u$ is a path completing the chordless cycle.

\textup{(b)} 
\[ \xymatrix{&\cdot\ar[d]_\zg\\
\cdot\ar@/^16pt/@{.>}[ru]^v&
k\ar[l]^\za\ar[r]_\zb&
\cdot\ar[lu]_\zs\ar@/_16pt/@{<.} [lu]_u
}\]
 where $\za,\zb$ are boundary arrows, $u$ is a path that consists entirely of boundary arrows, and $v$ is a path completing the chordless cycle.  Moreover, $v$ is not a boundary arrow. 
\end{lemma}

\begin{proof}
 (a)   Locally, the mutated quiver  $\mu_k Q$ is one of the following  
  \[ \xymatrix{&k\ar@{<-}[rd]^{\overline{\zb} }\\
\cdot\ar@/^0pt/[rr]^{[\za\zb]}\ar@{<-}[ru]^{\overline{\za}}&&\cdot\ar@/^10pt/@{.>} [ll]^u} \qquad \qquad
\xymatrix{&k\ar@{<-}[rd]^{\overline{\zb}}\\
\cdot\ar@{<-}[ru]^{\overline{\za}}&&\cdot}  \] 
where the left picture corresponds to the case where the path $u$ has length at least 2 and the right picture  to the case where $u$ is an arrow. 
In both cases, the quiver $\mu_k Q$ satisfies the conditions of Definition~\ref{def Q}. 

Now we show that the mutation of the potential also satisfies our definition. We need to treat the case where the path $u$ is a single arrow separately. Suppose first that $u$ is not an arrow. Then we write the potential  as
$W=\pm \za\zb u+ W'$, where none of the terms of $W'$ goes through the vertex $k$, hence $W'$ does not change under the mutation. 
Therefore the mutated potential is \[\mu_k W= 
\pm \,[\za\zb] u \,\pm \,[\za\zb]\zbbar\zabar +W'.\]
 To match our sign conventions in Definition~\ref{def B}, it suffices to change bases by replacing $\zabar$ by $-\zabar$, so we get the desired potential
\[\pm\, [\za\zb] u \,\mp\,[\za\zb]\zbbar\zabar +W'.\]
On the other hand, if $u$ is an arrow then we can write the potential
as
$W=\pm(\za\zb u-uv)+W''$, where
$v$ is the unique path forming the second chordless cycle with $u$ and $W''$ does not contain $\za,\zb$ or $u$.
Then the mutated potential is
\[\mu_k W=\pm([\za\zb]u +[\za\zb]\zbbar\zabar - uv)+W''.\]
Note that the first term on the right hand side is a 2-cycle. 
Using the relations $[\za\zb]=v$ and $u=-\zbbar\zabar$, which are obtained from the cyclic derivatives in the arrows $u$ and  $[\za\zb]$, we see that this potential is equivalent to the desired potential $\mp\zbbar\zabar v+W''$.

Next we want to show that the mutation is a derived equivalence. Note that the weight of $\overline{\za}$   in $\mu_k Q$ is opposite to the weight of $\zb$ in $Q$, independent of the length of $u$.  Indeed, the cycle path of $\zabar$ is
\[ \mathfrak{c}(\zabar) = \left\{
\begin{array}{ll}
  \zabar\,[\za\zb]\,\zd\dots &\textup{if the length of $u$ is at least 2 and $\zd$ is the first arrow in $u$;}      \\
  \zabar\ze \dots &\textup{if the length of $u$ is one, and $\ze$ is the first arrow in $v$,}
\end{array}
\right.
\]
and the cycle path of $\zb$ in $Q$ is 
\[ \mathfrak{c}(\zb) = \left\{
\begin{array}{ll}
  \zb\,\zd\dots &\textup{if the length of $u$ is at least 2 and $\zd$ is the first arrow in $u$;}      \\
  \zb \, u\,\ze \dots &\textup{if the length of $u$ is one, and $\ze$ is the first arrow in $v$.}
\end{array}
\right.
\]
Hence, since $\wt(\zb)=2$, we have
$\wt(\overline{\za})=1.$
By Theorem~\ref{thm BHL}, we conclude that the mutation at $k$ is a derived equivalence because
  \begin{itemize}
\item [(i)] before mutating, there is precisely one arrow ending in $k$ and, if $v$ is a nonzero path in $Q$ starting at $k$, then  Lemma~\ref{lem 1 dual} implies that $\za v\ne 0$, since $\cwt(\za)=1$.
\item[(ii)] after mutating, there is precisely one arrow starting in $k$ and, if $v$ is a nonzero path in $\mu_k Q$ ending at $k$,  then Lemma~\ref{lem 1} implies that the composition  $v\overline{\za}\ne 0$ because $\wt(\overline{\za})=1$. \ 
\end{itemize}
It remains to check that the total weight remains unchanged. 
By our assumptions, we have
\[\wt(\za) = 2,\ \cwt(\za)=1,\  \wt(\zb)=2, \] 
while on the other hand
\[\wt(\overline{\za}) = 1,\ \cwt(\overline{\za})=2,\  \cwt(\overline{\zb})=\cwt(\za)+1=2,\]
and the result follows.
Note that the cocycle path of $\zb$ is equal to the cycle path of $\za$, so the contribution of $\cwt(\zb)$ to the total weight of $Q$ is already counted in $\wt(\za)$. Similarly, the cycle path $\zbbar$ is equal to the cocycle path of $\zabar$, so the contribution of $\wt(\zbbar)$ to the total weight of $\mu_k Q$ is already counted in $\cwt(\zabar)$.

(b) 
 After mutating at $k$, the quiver is one of the following
 
 \[\xymatrix{
 &\cdot\ar@{<-}[d]^{\zgbar}\\
\cdot\ar@{<-}[ru]^(0.7)*!/u-16pt/{\rotatebox{45}{\scriptsize$[\zg\za]$}}\ar@/^16pt/@{.>}[ru]^v&
k\ar@{<-}[l]^{\zabar}\ar@{<-}[r]_{\zbbar}&
\cdot\ar@/_16pt/@{<.} [lu]_u
 }
 \qquad\qquad
 \xymatrix{
 &\cdot\ar@{<-}[d]^{\zgbar}\\
\cdot&
k\ar@{<-}[l]^{\zabar}\ar@{<-}[r]_{\zbbar}&
\cdot\ar@/_16pt/@{<.} [lu]_u
 }
 \] 
 where the left picture corresponds to the case where the path $v$ has length at least 2 and the right picture  to the case where $v$ is an arrow.
 Denote the potential of the original quiver by
 \[ W= \pm( \zg\za v - \zg\zb\zs + \zs u) +W'.  
 \]
 Its mutation is 
 \[\begin{array}{rcl}
  \mu_k W&=& \pm( [\zg\za] v+ [\zg\za] \zabar\zgbar -[\zg\zb]\zs - [\zg\zb]\zbbar\zgbar + \zs u) +W'\\
  &\cong& \pm( [\zg\za] v+ [\zg\za] \zabar\zgbar -\zbbar\zgbar  u) +W',\\
  \end{array}
 \]
 where the last equivalence is the reduction of the potential obtained by replacing $\zs$ by $-\zbbar\zgbar$, which removes the 2-cycle $[\zg\zb]\zs$. To match our sign conventions in Definition~\ref{def B}, it suffices to change bases by replacing $\zabar $ and $\zbbar$ by their negatives. Moreover, if $v$ has length 1 then $\mu_k W$ can be reduced further to $\mp\,\zbbar\zgbar  u +W'$.
 
 Next we want to show that the mutation is a derived equivalence.
 Let $\zd$ denote the last arrow in the path $u$ and write $u=u'\zd$. Note that $u'$ is a non-constant path, since $Q$ does not contain 2-cycles.
If $w$ is a nonzero path in $\mu_k Q$ starting at $k$ such that $\zbbar w=0$ then $\zbbar$ and a prefix of $w$ must lie in a relation. Since $\zbbar$ is a boundary arrow, the only such relation is $\partial_\zd W=\zbbar\zgbar u'$.  Thus $w$ is equivalent to a path $\zgbar u' w'$. However, since $u$ is a boundary path, the only arrow starting at the endpoint of $u'$ is $\zd$. Then $w\ne 0$ implies $w=\zgbar u'$. In this situation, $\zabar w =\zabar \zgbar u'$ is a nonzero path.

We have thus shown that for every nonzero path $w$ starting at $k$ in $\mu_k Q$ there is an arrow $\ze=\zabar$ or $\zbbar$ such that $\ze w\ne 0$. 
Moreover, note that $\wt(\zb)=1$, since $u$ is a boundary path. Thus Lemma~\ref{lem 1} implies that 
before mutating, if $w$ is a nonzero path in $Q$ ending at $k$ then $w\zb\ne 0$.
This shows that $\mu_k$ is a derived equivalence by Theorem~\ref{thm BHL}.

To check that the total weight is preserved, note that in $Q$
\[\wt(\za)=a, \ \cwt(\za)=2,\ \wt(\zb)=1,\ \cwt(\zb)=b,\]
with $a,b\in\{1,2\}$, and in $\mu_k Q$
\[\wt(\zabar)=1, \ \cwt(\zabar)=b,\ \wt(\zbbar)=\wt(\za)=a,\ \cwt(\zbbar)=2,\]
independent of the length of $v$. Thus the sum of the weights and coweights of the arrows, and hence the total weight, remains unchanged.
 \end{proof}

\subsection{Singular equivalences}
Let $B=\textup{Jac}(Q,W)$ be a dimer tree algebra. In this subsection, we describe two specific situations where a local change in the quiver with potential produces a singular equivalent algebra.

\begin{lemma} \label{lem sing}
 If $Q$ contains the subquiver on the left below, then replacing that subquiver with the quiver on the right induces a singular equivalence on the corresponding algebras that preserves the total weight of the quivers. 
 \[\xymatrix@C30pt@R10pt{
1\ar[r]^\rho &2\ar@/_16pt/@{.>}[l]_v \ar[dd]^\zs\\
&&3\ar[lu]_\za\\
5\ar@{.>}[uu] ^u & 4\ar[l]^\zg \ar[ru]_\zb
}
\qquad\qquad
\xymatrix@C30pt@R10pt{
1\ar[r]^\rho &2\ar@/_16pt/@{.>}[l]_v \ar@{<-}[dd]^{\zsbar}\ar[ldd]_(0.38)*!/u-16pt/{\rotatebox{45}{\scriptsize$[\zs\zg]$}}\\
&&3'\ar@{<-}[lu]_{\zebar}\\
5\ar@{.>}[uu] ^u & 4\ar@{<-}[l]^{\zgbar} \ar@{<-}[ru]_{\zdbar}
}
\]
Here $\za,\zb,\zg$ are boundary arrows and $u,v$ are paths that complete the chordless cycles. The path $u$ is allowed to be of length 0; in that case the vertices 1 and 5 are the same and the quivers become
 \[\xymatrix@C30pt@R10pt{
1\ar[r]^\rho&2\ar@/_16pt/@{.>}[l]_v \ar[ddl]_\zs\\
\\
4\ar[uu]^\zg\ar[r]^\zb & 3 \ar[uu]_\za
}
\qquad\qquad
\xymatrix@C30pt@R10pt{
1&2\ar@/_16pt/@{.>}[l]_v \ar@{<-}[ddl]_{\zsbar}\\
\\
4\ar@{<-}[uu]^{\zgbar}\ar@{<-}[r]^{\zdbar} & 3' \ar@{<-}[uu]_\zebar
}
\]

\end{lemma}
\begin{proof}
 The proof has three steps. First we perform a one-point coextension at vertex 4, see the left quiver in Figure~\ref{fig lem sing}. This is a singular equivalence by Corollary~\ref{cor Chen 1}. Call the new vertex $3'$.
   Then the mutation at vertex 4, which produces the quiver  shown in the right picture of Figure~\ref{fig lem sing}, is a derived equivalence by Theorem~\ref{thm BHL}, because
 \begin{itemize}
\item [(i)] before mutating, there is exactly one arrow ending at 4 and every nonzero path ending at 4 can be composed with the new arrow  $\zd$, thereby producing a nonzero path,
\item[(ii)] after mutating,  there is exactly one arrow starting at 4 and every nonzero path starting at 4 can be precomposed with the arrow $\zbbar$, producing a nonzero path. 
\end{itemize}
Finally, this last quiver is the one-point extension of the  quiver on the right in the statement of the lemma. Hence the two algebras are singular equivalent. 

\begin{figure}
\begin{center}
\[\xymatrix@C30pt@R10pt{
1\ar[r]^\rho &2\ar@/_16pt/@{.>}[l]_v \ar[dd]^\zs\\
&&3\ar[lu]_\za\\
5\ar@{.>}[uu] ^u & 4\ar[l]^\zg \ar[ru]^\zb\ar[r]_\zd&3'
}
\qquad\qquad
\xymatrix@C30pt@R10pt{
1\ar[r]^\rho&2\ar[dr]^{\zebar}\ar[ldd]_(0.38)*!/u-16pt/{\rotatebox{45}{\scriptsize$[\zs\zg]$}}\ar@/_16pt/@{.>}[l]_v \ar@{<-}[dd]^{\zsbar}\\
&&3'\\
5\ar@{.>}[uu] ^u & 4\ar@{<-}[l]^{\zgbar} \ar@{<-}[ru]^{\zdbar}\ar@{<-}[r]_{\zbbar}&3
}
\]
\caption{Proof of Lemma \ref{lem sing}. The quiver on  the left is the quiver of  the one-point coextension of $B$ at vertex 4. The quiver on the right is then obtained by mutation at vertex $4$. Note that the vertices $3$ and $3'$ have exchanged their position.}
\label{fig lem sing}
\end{center}
\end{figure}
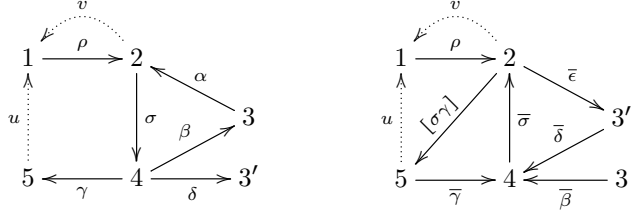

Clearly, the new quiver $Q'$ satisfies the conditions in Definition~\ref{def Q}. In order to check that the new potential satisfies the condition in Definition~\ref{def B}, let us denote the original potential by
\[W = \pm(\za\zs\zb-\rho\zs\zg u +\rho v)+\widetilde W.\]
The one-point coextensions do not change the potential. On the other hand, the mutation at 4 changes the potential to
\[ 
\begin{array}
 {rcl}
 W'&=& \zebar\zdbar\zsbar\, \pm([\zs\zb] \zbbar\zsbar+ \za[\zs\zb]   - {[\zs\zg]}\zgbar\zsbar   -\rho{[\zs\zg]} u+\rho v)+\widetilde W\\
 &\cong&  \zebar\zdbar\zsbar\, \pm(   - {[\zs\zg]}\zgbar\zsbar  -\rho{[\zs\zg]} u +\rho v)+\widetilde W,
 \end{array}\]
 where $\zebar $ is the arrow $[\zs\zd]$. If the path $u$ has length at least 1 then this potential does not contain any 2-cycles, and our sign conventions can be achieved by a change of bases, replacing $\zsbar$ (and $\zebar$ if necessary) by its negative. 
 
 On the other hand, if $u$ has length 0 then we also need to remove the 2-cycle $\rho[\zs\zg]$ from $W'$. This is done using the relations $[\zs\zg]=v$ and $\rho=-\zgbar\zsbar$, which are obtained from the derivatives in the arrows $\rho$ and $[\zs\zg]$, and we get the desired potential
 \[W'\ \cong \   \zebar\zdbar\zsbar\, - v\,\zgbar\zsbar +\widetilde W,\]
using the change of bases replacing $\zgbar$ by $-\zgbar$, if necessary.

To show that the total weight is preserved under this operation, we compute the weights and coweights in $Q$
\[ \wt(\za) =1, \ \wt(\zb)=2,\ \cwt(\zb)=a,\ \wt(\zg)=b,\ \cwt(\zg)=1\]
with $a,b \in\{1,2\}$. On the other hand, in the new quiver, we have
\[ \wt(\zebar) =2, \ \cwt(\zebar)=1,\ \wt(\zdbar)=b,\ \wt(\zgbar)=1,\ \cwt(\zgbar)=a\]
where $a$ and $b$ are the same as above. This completes the proof.
\end{proof}

\subsubsection{Removing a 3-cycle} 
In this subsection, we study when the removal of a boundary 3-cycle induces a singular equivalence.
\begin{lemma}\label{lem 3cycle}
 If $\za,\zb$ are boundary arrows that lie in a 3-cycle $\za\zb\zg$, and $b,b'$ are paths such that $b\za\ne 0$, $\zb b'\ne 0$ and  $b\za\zb b'=0$
then $\za\zb b'=0$ or $b\za\zb=0$.
\end{lemma}
\begin{proof} If $\zg$ is a boundary arrow then $Q$ is a single 3-cycle and the conclusion follows. Suppose now that $\zg$ is not a boundary arrow.
 Let $\ze_1 \zg \zd_1 u$ be the other chordless cycle at $\zg$, where $u$ is a path of length $\ge 0$, see Figure~\ref{fig lem 3cycle}. 
 \begin{figure}
\begin{center}
\[\xymatrix@R15pt@C35pt{
&&&k\ar[rd]^\zb 
\\
&&\cdot\ar[ru]^\za \ar[d]^{\zd_1} 
&&\cdot \ar@{.>}[rd]^{v_1} 
\ar[ll]_{\zg}\\
&\cdot \ar@{.>}[ru] ^{u_1} \ar[d]_{\zd_3}
& \cdot \ar[l]_{\zd_2} \ar@{.>}[rr]^u
&&\cdot\ar[u]^{\ze_1} \ar@{.>}[rd]_{v_2}
&\cdot \ar[l]_{\ze_2}
\\
\dots&\cdot\ar@{.>}[ru]_{u_2}
&&&& \cdot\ar[u]_{\ze_3} 
&\dots
}\]
\caption{Proof of Lemma \ref{lem 3cycle}.}
\label{fig lem 3cycle}
\end{center}
\end{figure}
The figure also shows the beginning of the cycle path $\mathfrak{c}(\zb)=\zb\zg\zd_1\zd_2\zd_3\dots$ and the end of the cocycle path $\overline{\mathfrak{c}}(\za)=\dots\ze_3\ze_2\ze_1\zg\za$, as well as the paths $u_i,v_i$ that complete the chordless cycles along these paths.
 
 Let $\widetilde{b}$ denote an arbitrary path equivalent to $b$ and $\widetilde{b'}$ a path equivalent to $b'$. First we show that $\widetilde{b}\za\zb \widetilde{b'}$ does not contain a cycle. Indeed, $\widetilde{b}\za$ and $\zb \widetilde{b'}$ do not contain cycles because they are nonzero paths. If the composition  $\widetilde{b}\za\zb \widetilde{b'}$ contains a cycle, then $\widetilde{b}$ and $\widetilde{b'}$ must share a vertex,
and since the dual graph is a tree, the only possible ways are
\begin{enumerate}
\item the path $\widetilde{b}$ or $\widetilde{b'}$ uses $\zg$ (but this is impossible, because $\zg\za=\zb\zg=0$).

\item  one of $\widetilde{b},\widetilde{b'}$ contains a segment that runs antiparallel to the path $u$. 
In this case, either $\widetilde{b}$ is equivalent to a path that contains $\zd_2 u_1$ as a subpath, which is impossible because $\zd_2 u_1= u \ze_1 \zg$ and $\zg\za =0$; 
or $\widetilde{b'}$ is equivalent to a path that contains $v_1 \ze_2$, which is impossible because $v_1 \ze_2=\zg \zd_1 u$ and $\zb\zg=0$.  


\item $u$ is a constant path and $\widetilde{b},\widetilde{b'}$ both contain its vertex. Then again $\widetilde{b}$ is equivalent to a path that contains $\zd_2 u_1$, and we have the same contradiction as in case (2) above.

\end{enumerate}
 
 So 
 $\widetilde{b}\za\zb \widetilde{b'}$ does not contain a cycle, for any paths $\widetilde{b}\cong b$ and $\widetilde{b'}\cong b'$.
 Also note that $\widetilde{b'}$ cannot contain $v_1\ze_2$ as a subpath, because $\zb v_1 \ze_2=\zb\zg\zd_1 u=0$, since $\zb\zg=0$. 
 Then $b\za\zb b'=0$ implies that there is a path $w$ that is equivalent to $b\za\zb b'$ and there exists a boundary arrow $\eta$ such that $w$ contains $\partial_\eta W$ as a subpath. 
 
 We first show that the path $w$ must contain either $\za\zb$ or $u$ as a subpath. Indeed, while we may replace $\za \zb$ by $\zd_1 u \ze_1$, and also $u_1 \zd_1$ by $\zd_3 u_2$, the resulting paths $u_2$ and $u$ are not in a relation because otherwise there would be a cycle in the dual graph formed by the chordless cycles around the vertex $s(u)$. A similar argument at the vertex $t(u)$ shows that we cannot remove both $u$ and $\za \zb$  from the path $w$.
 Thus we need one of the paths $\za\zb$ or $u$ to go from the left part of the quiver to the right.
 
Also note that the boundary arrow $\eta$ cannot be $\za $ or $\zb$ because, otherwise,  $b\za\zb b'$ would contain the cycle $(\partial_\za W) \za$ or $\zb(\partial _\zb W)$. 

  Furthermore $\eta$ cannot be an arrow of $u$,
because if $u=u'\eta u''$ then $\partial_\eta W =u'' \ze_1\zg\zd_1u'$ 
and $w $ must go through the starting point $s(u)$ of $u$ first in order to get to the starting point of $u''$.
Then $w$ would contain a subpath $w'=u\ze_1\zg\zd_1u$, since it has to reach $t(u)$ in order to join the path $b'$; otherwise $t(u)$ would be an interior point of $Q$, contradicting the assumption that the dual graph is a tree. 
Moreover, $w$ must be equivalent to a path that contains $\zd_1w'$ or $w'\ze_1$ as a subpath, because $w$ is equivalent to the path $b\za\zb b'$. Thus $w'\cong \za\zb\zg\zd_1 u$ and hence $b'$ is equal to a path using $\zg$, which is a contradiction to (1).

Let $Q(u)$ be the subquiver given as the connected component of $Q\setminus\{s(u),t(u)\}$ containing the remaining vertices of $u$.
Then the same argument also shows that $\eta$ cannot be an arrow in the subquiver $Q(u)$.

Finally, let $Q(\zd)$ and $Q(\ze)$ be the two connected components of $Q\setminus(Q(u)\cup\{k,\zg\}).$ We have shown that the arrow $\eta$ lies in $Q(\zd)$ or in $Q(\ze)$. In the first case, we see that the path $b\za\zb$ is equivalent to a path that contains the same zero relation $\partial_\eta W$ because the dual graph is a tree. Thus $b\za\zb=0$. In the second case, the path  $\za\zb b'$ is equivalent to a path that contains  $\partial_\eta W$. Thus $\za\zb b'=0$. 
\end{proof}

\begin{prop}
\label{prop 3cycle}
Let $B=\textup{Jac}(Q,W)$ be a dimer tree algebra. Suppose $Q$ contains a subquiver of the form
\[\xymatrix{&k\ar[rd]^\zb 
\\
\cdot\ar[ru]^\za  
&&\cdot  
\ar[ll]_{\zg}\\
}\]
with $\za,\zb$ boundary arrows. Let $(Q',W')$ be the quiver with potential obtained from $(Q,W)$ by removing the vertex $k$ and the arrows $\za$ and $\zb$ from the quiver and the cycle $\za\zb \zg$ from the potential. Let $B'=\textup{Jac}(Q',W')$ be the corresponding Jacobian algebra.
If  $\cwt(\za)=1$ and $\wt(\zb)=1$ then the algebras $B$ and $B'$ are singular equivalent and both have the same total weight.
\end{prop}

\begin{proof} Our first step is to use
 Theorem~\ref{thm chen} to show that $B$ and $B'$ are singular equivalent if and only if for all paths $b,b'$ such that $b\za$ and $\zb b'$ are nonzero paths, the composition $b\za\zb b'$ is nonzero. 
 Indeed, we can write $B$ as a matrix algebra
 \[B= 
\begin{pmatrix}
 A& A \za \\ \zb A &\kb
\end{pmatrix}\]
where $A=B/Be_k B$ is the quotient of $B$ by the two-sided ideal generated by all paths that contain the vertex $k$,  $A\za$ is the  left $A$-module generated by $\za$, and $\zb A$ is the right $A$-module generated by $\zb$. The multiplication can be seen in matrix form as shown below, where we use  the fact that $\zb c_1 b_2 \za=0$, because every non-constant cyclic path is zero in $B$, in the lower right corner.
\[
\begin{pmatrix}
 a_1& b_1\za \\ \zb c_1 &\zl_1
\end{pmatrix}
\begin{pmatrix}
 a_2& b_2\za \\ \zb c_2 &\zl_2 
\end{pmatrix}
=
\begin{pmatrix}
 a_1a_2 +b_1\za\zb c_2& a_1b_2\za +\zl_2 b_1\za \\ 
 \zb c_1 a_2+\zl_1\zb c_2  &\zl_1\zl_2
\end{pmatrix}
\]
where $a_i,b_i,c_i\in A$ and $\zl_i\in \kb$.

Let us now check that in this situation the conditions of Theorem~\ref{thm chen} are satisfied.
 The morphism $\phi\colon A\za\otimes_\kb\zb A \to A$ is given by concatenation
 $\phi (b\za\otimes \zb c) =b\za\zb c$. We have $(\im\, \phi) (A\za) = 0$ and $(\zb A)(\im\, \phi)=0$, because every non-constant cyclic path is zero in $B$. The only remaining condition is that $\phi $ is mono. 
 Since any two nonzero parallel paths are equal in $B$ it suffices to show that if $b \za$ and $\zb c$ are nonzero paths then $b\za\zb c$ is a nonzero path. 
 
  Because of Lemma~\ref{lem 3cycle}, it suffices to show that $b\za\zb$  and $\za\zb c$ are both nonzero. For the first path this follows directly from Lemma~\ref{lem 1}, because $b\za\ne 0$ and $\wt(\zb)=1$.
For the second path, it follows from  Lemma~\ref{lem 1 dual}, because $\zb b'\ne 0$ and $\cwt(\za)=1$.

Therefore Theorem~\ref{thm chen} implies that $B$ and $A/\im\, \phi$ are singular equivalent. Note that $\im\, \phi$ is the ideal in $A$ generated by all paths that have $\za\zb$ as a subpath. In particular, if we let $ u\ne \za\zb$ be the unique other path  in $Q$ that 
forms a chordless cycle $\zg u$ with $\zg$, then in $B$ the path $u$ is equal to the path $\za\zb$. In particular, the path $u$ is nonzero in $A$, and it is a generator of the ideal $\im\, \phi$. Thus in $A/\im\, \phi$ we have the additional relation $u=0$ which is equal to the derivative $\partial_\zg W'$ of the potential $W'$. Thus $B'\cong A/\im\, \phi$. 

To show that the total weight remains unchanged, observe that in $Q$ we have 
\[\wt(\za) + \wt(\zb)+\cwt(\za)=2+1+1=4.\] 
In $Q'$, the arrow $\zg$ becomes a boundary arrow with weight $\wt(\zg)=\wt(\zb)+1=2$  
and coweight $\cwt(\zg)=\cwt(\za)+1=2$. Thus the total weight is preserved. 
We point out that we do not  need to add the coweight of $\zb$ in this situation, because its contribution to the total weight is already counted in the weight of $\za$, since the cycle path of $\za$ is equal to the cocycle path of $\zb$ (both are $\za\zb$).
\end{proof}

\begin{remark}
The converse of  Proposition~\ref{prop 3cycle} is also true, meaning that if the weight conditions do not hold then the algebras are not singular equivalent. This can be proved using the main result of this paper. Indeed, if $\wt(\zb)=2$ in $Q$ then $\wt(\zg)=1$ in $Q'$, and so the total weight changes. Dually, if $\cwt(\za)=2$ in $Q$ then the first arrow in the cocycle path of $\za$ will change weight from $2$ to 1, and again the total weight is not preserved.
\end{remark}

\subsection{The algorithm}
We are now ready to prove the main result of this section. It generalizes Lu's theorem on simple polygonal-tree algebras.
\begin{thm}\label{thm algo}
Let $B=\textup{Jac}(Q,W)$ be a dimer tree algebra.  Denote by $2N$ the total weight of $Q$, and let $S$ be a polygon with $2N$ vertices. Then there is a triangle equivalence of categories 
\[\scmp \,B \cong \diags.\]
\end{thm}
\begin{proof}
We proceed by induction on the number of chordless cycles in $Q$. If this number is 1, then $B$ is a selfinjective cluster-tilted algebra of Dynkin type $\mathbb{D}_n$ and $\scmp\,B=\underline{\textup{mod}}\,B$. It is well-known that this category is equivalent to the category of 2-diagonals of a polygon with $2n$ vertices. For convenience, we include a proof in the appendix. On the other hand,  $Q$ has $n$ arrows, each a boundary arrow and each having weight 2. So the total weight is $2n$. Thus our result holds in this case.

Now suppose $Q$ has more than one chordless cycle. Since the dual graph is a tree, we may choose a chordless cycle $C_0$ that has exactly one interior arrow $\zg$. Denote by $1,2,\dots,m$ the vertices of $C_0$ in order, where $\zg$ is the arrow from $1$ to $2$. Let $\za$ denote the arrow that follows $\zg$ in $C_0$ and let $\zb$ be the arrow following $\za$, see the first quiver in Figure~\ref{fig algo 1}. Let $u$ be the unique path such that $C_1=\zg u$ is the second chordless cycle that contains $\zg$.
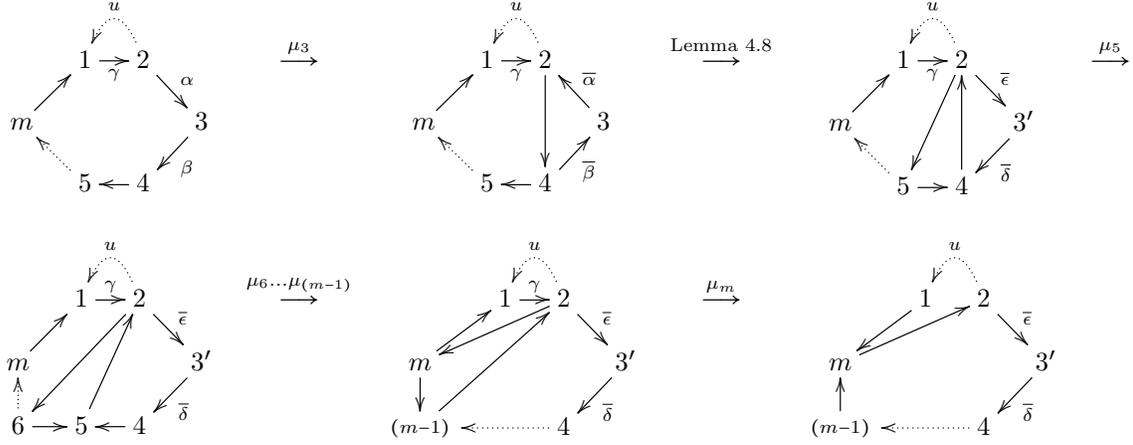
\begin{figure}
\begin{center}
\[
\begin{array}
 {ccccccc}
\xymatrix@C\quiversize@R\quiversize{
&1\ar[r]_\zg & 2 \ar[rd]^\za\ar@{.>}@/_16pt/[l]_u
\\
m\ar[ru] &&&3\ar[dl]^\zb
\\
&5\ar@{.>}[lu]&4\ar[l]
}
& \stackrel{\mu_3}{\longrightarrow}
&
\xymatrix@C\quiversize@R\quiversize{
&1\ar[r]_\zg & 2 \ar@{<-}[rd]^{\zabar}\ar[dd]\ar@{.>}@/_16pt/[l]_u
\\m\ar[ru] &&&3\ar@{<-}[dl]^{\zbbar}
\\
&5\ar@{.>}[lu]&4\ar[l]
}
&\stackrel{\textup{Lemma \ref{lem sing}}}{\longrightarrow}&
\xymatrix@C\quiversize@R\quiversize{
&1\ar[r]_\zg & 2 \ar[ldd]\ar@{->}[rd]^{\zebar}
\ar@{.>}@/_16pt/[l]_u
\\
m\ar[ru] &&&3'\ar@{->}[ld]^{\zdbar}
\\
&5\ar[r]\ar@{.>}[lu]&4\ar[uu]
}
&\stackrel{\mu_5}{\longrightarrow}
\\
\\
\xymatrix@C\quiversize@R\quiversize{
&1\ar[r]^\zg & 2 \ar[ddll]\ar@{<-}[ldd]\ar@{->}[rd]^{\zebar}
\ar@{.>}@/_16pt/[l]_u
\\
m\ar[ru] &&&3'\ar@{->}[ld]^{\zdbar}
\\
6\ar[r]\ar@{.>}[u]&5\ar@{<-}[r]&4
}
&\stackrel{\mu_6\dots\mu_{(m-1)}}{\longrightarrow}
&
\xymatrix@C\quiversize@R\quiversize{
&1\ar[r]^\zg & 2 \ar[dll]\ar@{->}[rd]^{\zebar}
\ar@{.>}@/_16pt/[l]_u
\\
m\ar[ru]\ar[d] &&&3'\ar@{->}[ld]^{\zdbar}
\\
{\begin{smallmatrix} (m-1)\end{smallmatrix}}\ar[rruu]&&4\ar@{.>}[ll]
}

&\stackrel{\mu_m}{\longrightarrow}
&
\xymatrix@C\quiversize@R\quiversize{
&1 & 2 \ar@{<-}[dll]\ar@{->}[rd]^{\zebar}
\ar@{.>}@/_16pt/[l]_u
\\
m\ar@{<-}[ru]\ar@{<-}[d] &&&3'\ar@{->}[ld]^{\zdbar}
\\
{\begin{smallmatrix} (m-1)\end{smallmatrix}}&&4\ar@{.>}[ll]
}

\end{array}
\]
\caption{Proof of Theorem \ref{thm algo}. A sequence of equivalences that reduces the length of the cycle by 1.}
\label{fig algo 1}
\end{center}
\end{figure}
We consider two cases depending on the coweight of the arrow $\za$.

(1) Assume first that $\cwt(\za)=1$. The idea is to successively shorten the chordless cycle $C_0$ until it will be completely absorbed into the cycle $C_1$. Assume first that $m>4$. Then the weight of the arrow $\zb$ is 2. We will now apply a sequence of derived and singular equivalences that preserve the total weight in order to reduce the length of the cycle $C_0$ from $m$ to $m-1$. The corresponding quivers are shown in Figure~\ref{fig algo 1}, and the steps are justified below.

Step 1.  Apply the mutation at 3, which is a derived equivalence that preserves the total weight, by  Lemma~\ref{lem der}(a).

Step 2. Apply the singular equivalence of Lemma \ref{lem sing}.

Steps 3, 4 and 5. Apply the mutations at $5, 6, 7, \dots, m$. They are all derived equivalences,  by  Lemma~\ref{lem der}(b). Note that the last mutation at $m$ makes the arrow $\zg$ disappear.

This sequence of equivalences has reduced the $m$-cycle $C_0$ to an $(m-1)$-cycle. The cocycle path of the arrow $\zebar$ in the new quiver is $\overline{\mathfrak{c}}(\zebar)=
1\to m\to 2\to 3'$.
Thus $\cwt(\zebar)=1$, and we can repeat the sequence of equivalences again.

After $m-4$ rounds, we have reduced our cycle to length 4, and we obtain a subquiver isomorphic to  the first quiver in  Figure~\ref{fig algo 2}. Moreover the coweight of the arrow $\za$ is 1. 
We now apply another sequence of derived and singular equivalences that preserve the total weight. The corresponding quivers are shown in Figure~\ref{fig algo 2}, and the steps are justified below.
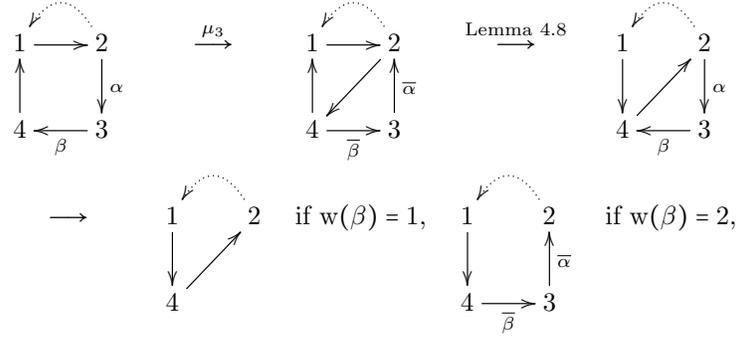
\begin{figure}
\begin{center}
\[
\begin{array}
 {ccccccc}
\xymatrix@C20pt@R20pt{
1\ar[r] & 2 \ar[d]^\za\ar@{.>}@/_16pt/[l]
\\
4\ar[u] &3\ar[l]^\zb
}
&\stackrel{\mu_3}{\longrightarrow}
&
\xymatrix@C20pt@R20pt{
1\ar[r] & 2 \ar@{<-}[d]^{\zabar}\ar[ld]\ar@{.>}@/_16pt/[l]
\\
4\ar[u] &3\ar@{<-}[l]^{\zbbar}
}
&\stackrel{\textup{Lemma~\ref{lem sing}}}{\longrightarrow}
&
\xymatrix@C20pt@R20pt{
1& 2 \ar[d]^{\za}\ar@{<-}[ld]\ar@{.>}@/_16pt/[l]
\\
4\ar@{<-}[u] &3\ar[l]^{\zb}
}
\\
\\
{\longrightarrow}
&
\xymatrix@C20pt@R20pt{
1& 2 \ar@{<-}[ld]\ar@{.>}@/_16pt/[l]
\\
4\ar@{<-}[u]}  &\textup{if $\wt(\zb)=1$,}
&
\xymatrix@C20pt@R20pt{1& 2 \ar@{<-}[d]^{\zabar}\ar@{.>}@/_16pt/[l]
\\
4\ar@{<-}[u] &3\ar@{<-}[l]^{\zbbar}
} &\textup{if $\wt(\zb)=2$,}
\end{array}
\]
\caption{Proof of Theorem \ref{thm algo}. The second sequence of equivalences}
\label{fig algo 2}
\end{center}
\end{figure}

Step 6. Apply  the mutation at $3$, which is a derived equivalence, by Lemma \ref{lem der}(a), since $\cwt(\za)=1$ and $\wt(\zb)=2$. 

Step 7. 
Apply the singular equivalence of Lemma~\ref{lem sing} in the special case where the path $u$ has length 0.

Step 8. Now our cycle has length 3 and $\cwt(\za)=1$. If $\wt(\zb)=1$, we remove $\za,\zb$ and the vertex 3, which is a singular equivalence, by Proposition~\ref{prop 3cycle}. The resulting quiver has one less cycle, and so we are done by induction.

On the other hand, if $\wt(\zb)=2$, we apply the mutation at 3, which is a derived equivalence, by Lemma~\ref{lem der}(a). 
Again the resulting quiver has one less cycle, and we are done by induction.

This completes the proof in the case where the arrow $\za$ in the first quiver in Figure~\ref{fig algo 1} has coweight one.

(2) Now assume that the arrow $\za$ has coweight 2. This case is illustrated in Figure~\ref{fig algo 3}.
If the arrow $\ze$ at the other end of the cycle has weight 1, then we can use the dual argument of case (1). So we may assume without loss of generality that $\wt(\ze)=2$. 
We will apply a sequence of derived and singular equivalences that will reduce to case (1).   The corresponding quivers are shown in Figure~\ref{fig algo 3}, and the steps are justified below.
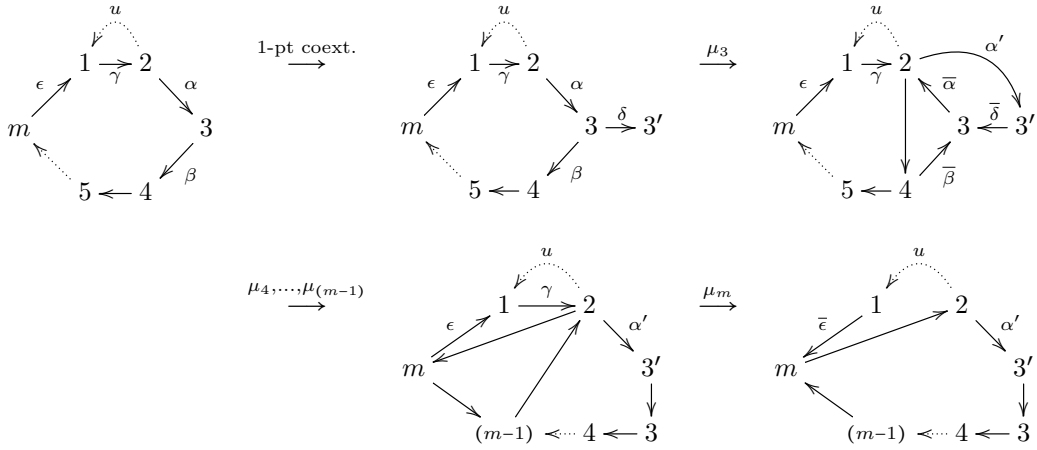
\begin{figure}
\begin{center}
\[
\begin{array}
 {ccccccc}
\xymatrix@C12pt@R12pt{
&1\ar[r]_\zg & 2 \ar[rd]^\za\ar@{.>}@/_16pt/[l]_u
\\
m\ar[ru]^\ze &&&3\ar[dl]^\zb
\\
&5\ar@{.>}[lu]&4\ar[l]
}
& \stackrel{\textup{1-pt coext.}}{\longrightarrow}
&\xymatrix@C\quiversize@R\quiversize{
&1\ar[r]_\zg & 2 \ar[rd]^\za\ar@{.>}@/_16pt/[l]_u
\\
m\ar[ru]^\ze &&&3\ar[dl]^\zb&3'\ar@{<-}[l]_\zd
\\
&5\ar@{.>}[lu]&4\ar[l]
}
& \stackrel{\mu_3}{\longrightarrow}
&\xymatrix@C\quiversize@R\quiversize{
&1\ar[r]_\zg & 2\ar[dd] \ar@{<-}[rd]^\zabar\ar@{.>}@/_16pt/[l]_u\ar@/^16pt/[drr]^{\za'}
\\
m\ar[ru]^\ze &&&3\ar@{<-}[dl]^\zbbar&
3'\ar[l]_\zdbar
\\
&5\ar@{.>}[lu]&4\ar[l]&
}
&
\\
\\
&
\stackrel{\mu_4,\dots,\mu_{(m-1)}}{\longrightarrow}
&
\xymatrix@C\quiversize@R\quiversize{
&1\ar[r]^\zg & 2\ar[dll] \ar@{<-}[ldd]\ar@{.>}@/_16pt/[l]_u\ar@/^0pt/[dr]^{\za'}
\\
m\ar[ru]^\ze\ar[rd] &&&3'\ar[d]
\\
&
{\begin{smallmatrix}
 (m-1)
\end{smallmatrix}}
&4\ar@{.>}[l]&3\ar[l]
}
&\stackrel{\mu_m}{\longrightarrow}
&
\xymatrix@C\quiversize@R\quiversize{
&1 & 2\ar@{<-}[dll] \ar@{.>}@/_16pt/[l]_u\ar@/^0pt/[dr]^{\za'}
\\
m\ar@{<-}[ru]^\zebar\ar@{<-}[rd] &&&3'\ar[d]
\\
&
{\begin{smallmatrix}
 (m-1)
\end{smallmatrix}}
&4\ar@{.>}[l]&3\ar[l]
}

\end{array}
\]
\caption{Proof of Theorem \ref{thm algo}, second case.}
\label{fig algo 3}
\end{center}
\end{figure}

Step 1. Apply a one-point coextension at vertex 3. This is a singular equivalence by Corollary~\ref{cor Chen 1}.  Call the new vertex $3'$ and the new arrow $\zd$.

Step 2. Apply the mutation at vertex 3. This is a derived equivalence by Theorem~\ref{thm BHL}, because
\begin{itemize}
\item [(i)] before mutating, there is exactly one arrow ending at 3, and if $v$ is a nonzero path ending in $3$ then $v\zd\ne 0$;
\item[(ii)] after mutating, there is exactly one arrow starting at 3. Moreover, note that the arrow $\zbbar\colon4\to 3$ has coweight $\cwt(\zbbar)=1$, because before the mutation $\cwt(\za)=2$. Thus if $v$ is a nonzero path starting at $3$ then $\zbbar v\ne 0$, by Lemma~\ref{lem 1 dual}.  
\end{itemize}

Step 3. Apply mutations at $4, 5, \dots, (m-1)$. These are all derived equivalences, by Lemma~\ref{lem der}(b). The resulting quiver is the first quiver in the second row of Figure~\ref{fig algo 3}.

Step 4. Apply the mutation at $m$, which also is a derived equivalence by  Lemma~\ref{lem der}(b).
\smallskip

Note that in the resulting quiver, the first arrow $\za'$ of our chordless cycle now has coweight 1, because $\zebar$ is a boundary arrow. 
Since $\cwt(\za')=1$, we can use the argument of case (1) and the proof is complete.
\end{proof}


\section{Construction of the checkerboard pattern}\label{sect 2}
Throughout this section, $B=\textup{Jac}(Q,W)$ is a dimer tree algebra. In this section, we show that, knowing that the stable syzygy category of $B$ is equivalent to the category of 2-diagonals of a polygon, there exists a unique checkerboard pattern on the polygon that realizes the homological structure of the category, see Theorem~\ref{thm poly}. 
\subsection{Extensions between syzygies and radicals of projectives}

In this subsection, we show that extensions between a syzygy $M$ and the radicals $\rad P(i)$ are determined by the projective resolution of $M$.   We start with a preliminary lemma about the radicals.

\begin{lemma}\label{lem:rad}
The radical $\rad P(i)$ is an indecomposable non-projective syzygy. 
\end{lemma}

\begin{proof}
Since $\rad P(i) = \Omega \,S(i)$, we see that $\rad P(i)$ is a syzygy.  Note that it is nonzero, because every vertex lies in an oriented cordless cycle in $Q$ so there are no sinks in $Q$.  Moreover, since every arrow of $Q$ lies in a relation, it follows that $\rad P(i)$ is non-projective.  
To show that $\rad P(i)$ is indecomposable consider the top of $\rad P(i)$.  If the top is simple, then we are done.  Otherwise, the top is semi-simple and isomorphic to $S(a_1)\oplus \dots \oplus S(a_k)$ for some $k\geq 2$, where $i\to a_j \to \dots \to b_j \to i$  and $i \to a_{j+1} \to \dots \to b_j \to i$ are oriented chordless cycles in $Q$ that have an arrow in common.   Then the representation $\rad P(i)$ is supported on the path \[a_j\to \dots \to b_j \leftarrow \dots \leftarrow a_{j+1}\] for $j = 1, \dots, k-1$, which shows that it is indecomposable.    
\end{proof}

Now, we consider extensions in the module category between syzygies and radicals of projectives.

\begin{lemma}\label{lem:ext}
Let $M$ be an indecomposable non-projective syzygy in $\textup{mod}\,B$, and let $P_0$ be the projective cover of $M$.   Then the following are equivalent.

\begin{itemize}
\item[(a)] $\textup{Ext}^1_B(M, \rad P(x))\not=0$. 
\item[(b)] $P(x)$ is a direct summand of $P_0$. 
\end{itemize}
\end{lemma}

\begin{proof}
Suppose (a). Consider the following diagram where the first row is a minimal projective resolution of $M$ in $\text{mod}\,B$. 
\[\xymatrix{\dots \ar[r]& P_2\ar[r]^{p_2} & P_1\ar[r]^{p_1}\ar[d]^{g}&P_0\ar[r]^{p_0}\ar@{-->}[ddl]^{h}& M \ar[r] & 0\\
&&\rad P(x) \ar[d]^{i}\\
&&P(x)
}\]
Let the map $g: P_1\to \rad P(x)$ induce a nonzero element of $\textup{Ext}^1_B(M, \rad P(x))$.    Then $gp_2=0$ and $g$ does not factor though the map $p_1$.  Let $i:\rad P(x)\to P(x)$ be the inclusion morphism of the radical into its corresponding projective module.  Since $M$ is a syzygy then $\textup{Ext}^1_B(M,  P(x))=0$.  In particular, the equivalence class of the map $ig$ is zero in $\textup{Ext}^1_B(M,  P(x))$, and so $igp_2\not=0$ or $ig$ factors through $p_1$.  Since $gp_2=0$, we must have that $ig$ factors through $p_1$.  Thus, there exists $h: P_0\to P(x)$ such that 
\[hp_1=ig.\]   
Now, suppose to the contrary that $P(x)$ is not a direct summand of $P_0$.  Then $h$ factors through $i$, so there exists $h':P_0\to \rad P(x)$ such that $ih'=h$.  Therefore, 
\[ig=hp_1=ih'p_1.\]  
Because $i$ is injective we conclude that $g=h'p_1$.  Therefore, $g$ factors through $p_1$ which is a contradiction to it being nonzero in $\textup{Ext}^1_B(M, \rad P(x))$.  This shows that (a) implies (b).

Now suppose (b).  Consider the following diagram. 
\[
\xymatrix{&&\text{ker}\,\pi'\ar[r] \ar@{^{(}->}[d]&\text{ker}\,\pi \ar@{^{(}->}[d]\\
0\ar[r] &\rad P(x) \ar[r] \ar@{=}[d]& X \ar[r] \ar@{->>}[d]^{\pi'}& M \ar[r]\ar@{->>}[d]^{\pi} & 0\\
0\ar[r]&\rad P(x)\ar[r] & P(x)\ar[r]^{p} & S(x)\ar[r] & 0
}
\]
The bottom row of the diagram is a short exact sequence ending in the simple module $S(x)$.   Since $P(x)$ is a direct summand of the projective cover of $M$, then $S(x)$ appears in the top of $M$.  Let $\pi: M\to S(x)$ be the corresponding projection.  The pullback of $S(x)$ along $\pi$ and $p$ yields a short exact sequence appearing in the second row of the diagram.   If this sequence is nonsplit, then we obtain a nonzero element of $\text{Ext}^1_B(M, \rad P(x))$ and we are done.  Otherwise, we have that $X\cong \rad P(x)\oplus M$. 

Since $\pi$ is surjective, then by the snake lemma $\pi'$ is also surjective and $\text{ker}\,\pi'\cong \text{ker}\,\pi$.  Moreover, we obtain another short exact sequence appearing in the second column of the diagram which ends in the projective $P(x)$.  This sequence must split and we conclude $X \cong \text{ker}\,\pi \oplus P(x)$.  

Hence, $\text{ker}\,\pi \oplus P(x) \cong \rad P(x)\oplus M$, and we have that $P(x)$ is a direct summand of $M$ or $\rad P(x)$. The former is not possible because $M$ is non-projective and indecomposable by assumption, and the latter is not possible because the algebra $B$ is finite dimensional.  This shows that (b) implies (a).
\end{proof}

Next, we show that the previous result also holds for extensions in the syzygy category, but first we need the following lemma. 

\begin{lemma}\label{lem:MN}
Let $M, N \in \scmp\,B$. Then $\textup{Ext}^1_B(M, N) = 0$ if and only if $\textup{Ext}^1_{\scmp\,B}(M,N)=0$.
\end{lemma}

\begin{proof}
The forward direction follows because every triangle in  $\scmp\,B$ lifts to a short exact sequence in $\text{mod}\,B$.   It remains to show the backward direction, so we suppose that  $\textup{Ext}^1_{B}(M,N)\not=0$.  Since $N\in \scmp\,B$, let $X = \Omega^{-1}N$ be its shift.  Hence, we have $\Omega X = N$ in $\scmp\,B$, and since $\Omega$ is the same in $\scmp\,B$ and $\text{mod}\,B$ we get a short exact sequence with $P$ a projective cover of $X$.
\[0\to N \to P \to X \to 0\]
Applying $\text{Hom}_B(M, - )$ to this sequence, we obtain
\[\text{Hom}_B(M,P)\to\text{Hom}_B(M,X)\to \text{Ext}^1_B(M,N)\to \text{Ext}^1_B(M,P).\]
 The last term above is zero because $M\in \scmp\,B$.  By assumption $\textup{Ext}^1_{B}(M,N)\not=0$, which implies that the map $\text{Hom}_B(M,P)\to\text{Hom}_B(M,X)$ is not surjective.  Since $P$ is the projective cover of $X$, this implies that not every morphism from $M$ to $X$ factors through a projective module.  Thus the lemma follows from the equation below.
\[0 \not= \underline{\text{Hom}}_B(M,X) \cong \text{Hom}_{\scmp\,B}(M,X)\cong \text{Ext}^1_{\scmp\,B} (M, \Omega X) = \text{Ext}^1_{\scmp\,B}(M, N).\]
\end{proof}

\begin{prop}\label{prop:ext}
Let $M$ be an indecomposable non-projective syzygy in $\textup{mod}\,B$, and let
\[\xymatrix{P_1\ar[r]&P_0\ar[r] & M\ar[r] & 0}\]
be the minimal projective resolution of $M$ in $\textup{mod}\,B$. Then 

\begin{itemize}
\item[(i)] $\textup{Ext}^1_{\scmp\,B}(M, \rad P(x))\not=0$ if and only if $P(x)\in\textup{add}\,P_0$.  
\item[(ii)] $\textup{Ext}^2_{\scmp\,B}(M, \rad P(x))\not=0$ if and only if $P(x)\in\textup{add}\,P_1$. 
\end{itemize}
\end{prop}

\begin{proof}
Part (i) follows directly from Lemmas~\ref{lem:ext} and ~\ref{lem:MN}. Part (ii) follows from (i) by replacing $M$ with $\Omega\, M$ since $\textup{Ext}^2_{\scmp\,B}(M, \rad P(x))\cong \textup{Ext}^1_{\scmp\,B}(\Omega\, M, \rad P(x))$. 
\end{proof}




As an immediate corollary, we obtain an explicit description of the projective presentation of $\rad P(i)$. 

\begin{corollary}\label{cor:rad}
The minimal projective presentation of the radical $\rad P(x)$ in $\textup{mod}\,B$ is 
\[\bigoplus_{j\to x \in Q_1} P(j)\to \bigoplus_{x\to i \in Q_1} P(i) \to \rad P(x) \to 0.
\]
\end{corollary}

\begin{proof}
It is easy to see that the projective cover of $\rad P(x)$ is as in the statement.  Next, we consider the second term in the projective resolution of $\rad P(x)$.  By \cite[Proposition 3.32]{SS} there exists an arrow $j\to x$ in the quiver $Q$ if and only if $\text{Ext}^1_{B}(\rad P(j), \rad P(x))\not=0$.  By Lemma~\ref{lem:ext} and Proposition~\ref{prop:ext}(i), this is equivalent to the statement that 
\[\text{Ext}^1_{\scmp\,B}(\rad P(j), \rad P(x))\not=0.\]  
Then 
\[\text{Ext}^1_{\scmp\,B}(\rad P(j), \rad P(x))\cong D\text{Ext}^2_{\scmp\,B}(\rad P(x), \rad P(j))\]
by the 3-Calabi-Yau property of $\scmp\,B$. Then Proposition~\ref{prop:ext}(ii) implies that the last term is nonzero if and only if $P(j)$ is a summand of the second term in the projective presentation of $\rad P(x)$.  This yields the desired result. 
\end{proof}

\subsection{Arrows and extensions}

In this section we study the connection between arrows in the quiver and extensions between the corresponding radicals. 

\begin{lemma}\label{lem:arrow}
If $i\to j$ is an arrow in $Q$, then $\textup{Ext}^1_{\scmp B}(\Omega \,\rad P(i), \rad P(j))=0$.
\end{lemma}

\begin{proof}
First we claim that 
\begin{equation} \label{eq:b}
\text{Hom}_B(\Omega \,\rad P(i), S(j))=0.
\end{equation}   
To show the claim consider a short exact sequence in $\text{mod}\,B$
\[ 0\to \Omega \,\rad P(i) \to P \to \rad P(i)\to 0.\]
Apply $\text{Hom}_B(-, S(j))$ to obtain the exact sequence 

\begin{multline*} 
0\to \Hom_B(\rad P(i), S(j))\to \Hom_B(P,S(j))\to \Hom_B(\Omega \,\rad P(i), S(j))\to \\
\to \Ext_B^1(\rad P(i), S(j)) \to \Ext_B^1(P, S(j))
\end{multline*}

The last term is zero since $P$ is projective.  Since there is exactly one arrow $i\to j$ then both vector spaces $\Hom_B(\rad P(i), S(j))$ and $\Hom_B(P,S(j))$ are 1-dimensional.  Therefore, we conclude that 
\begin{equation}\label{eq:a}
\Hom_B(\Omega \,\rad P(i), S(j)) \cong \Ext_B^1(\rad P(i), S(j)).
\end{equation}

By Corollary~\ref{cor:rad}, the second term in the projective resolution of $\rad P(i)$ is $\oplus P(x)$ where the sum runs over all arrows $x\to i$ in the quiver $Q$. Since by assumption there is an arrow $i\to j$ in $Q$, then there is no arrow $j\to i$, and so $P(j)$ is not a summand of the second term in the projective resolution of $\rad P(i)$.  This implies that $\Ext^1_B(\rad P(i), S(j))=0$ and completes the proof of \eqref{eq:b}.

Now, we want to show the following
\begin{equation}\label{eq:c}
\Ext^1_B(\Omega \,\rad P(i), \rad P(j))=0.
\end{equation}
Applying $\Hom_B(\Omega \,\rad P(i), -)$ to the short exact sequence 
\[ 0 \to \rad P(j)\to P(j) \to S(j)\to 0,\]
we obtain 
\[ \Hom_B(\Omega \,\rad P(i), S(j))\to \Ext^1_B(\Omega \,\rad P(i), \rad P(j))\to \Ext^1_B(\Omega \,\rad P(i), P(j)).\]
The first term above is zero by \eqref{eq:b} and the last term is zero because $\Omega \,\rad P(i)\in \scmp\,B$.  This shows \eqref{eq:c}, and the conclusion follows from Lemma~\ref{lem:MN}.
\end{proof}

\begin{prop}\label{prop:barrow} 
If $i\to j$ is a boundary arrow in $Q$, then 
\[\textup{Ext}^1_{\scmp B}(\Omega \,\rad P(i), \rad P(j)) =\textup{Ext}^1_{\scmp B}(\rad P(j), \Omega \,\rad P(i))  =0.\]
\end{prop} 

\begin{proof}
By Lemma~\ref{lem:arrow} it suffices to show that $\textup{Ext}^1_{\scmp B}(\rad P(j), \Omega \,\rad P(i))  =0$. This is equivalent to \[\underline{\text{Hom}}_B(\rad P(j), \rad P(i))=0\] which we show below 
using the particular structure of these quiver representations. 

Consider the local configuration around vertex $j$ shown below.  Here each dotted arrow is either a single arrow or a sequence of arrows such that each bounded region is an oriented chordless cycle in $Q$.  The arrow $i\to j$ is a boundary arrow as in the statement of the proposition, and, without loss of generality, we may assume that the other boundary arrow $j\to a_k$ at vertex $j$ is oriented away from $j$.  The other case follows similarly. 

\[
\xymatrix{a_k\ar@/^5pt/@{..>}[r]&b_k\ar[d] & a_2\ar@/^5pt/@{..>}[d]\ar@{.}[l]\\
i\ar[r] & j \ar[dl] \ar[dr]\ar[ur]\ar[ul]& b_2\ar[l]\\
a_0\ar@/^5pt/@{..>}[u] \ar@/_5pt/@{..>}[r] & b_1\ar[u] & a_1\ar@/^5pt/@{..>}[l] \ar@/_5pt/@{..>}[u]}
\]

Then the representation $\rad P(j)$ has the following structure, where each simple representation $S(a_r)$, $r=1,2,\dots,k,$ appears in the top of $\rad P(j)$.  Moreover, $\rad P(j)$  is supported on the paths $a_r \to \dots \to b_{r+1} \leftarrow \dots \leftarrow a_{r+1}$ for $r<k$. 

\[\xymatrix@R=4pt@C=4pt{&&&a_0 \ar[dl]\ar[dr]&&&& a_1 \ar[dl]\ar[dr]&&&& a_2 \ar[dl]\ar[dr]&&&& a_k\ar[dl] \ar@{..}[dr]\\
\rad P(j)= &&\udots &&\ddots\ar[dr]&&\udots\ar[dl]&&\ddots\ar[dr]&&\udots\ar[dl]&&\cdots\ar[dr]&&\udots\ar[dl]&& \\
&&&&&b_1 &&&& b_2 &&&& b_k &&&&\\
&&&&&\vdots &&&& \vdots &&&& \vdots &&&&}
\]

The representation $\rad P(i)$ has the following structure.  Note that the top of $\rad P(i)$ is not generally $S(j)$ and corresponds to all vertices that have an arrow from $i$, but here we only depict the relevant vertices.  Moreover, $\rad P(i)$ may not be supported on the path $\dots \leftarrow a_0\to \dots \to b_1$ if $i\to j \to a_0\to i$ is a 3-cycle and the arrow $a_0\to i$ is a boundary arrow.   We consider the most general situation below, because the proof in this special case follows similarly. 

\[
\xymatrix@R=4pt@C=4pt{&&&&&&&&&j \ar[drr]\ar[dll]\ar[drrrrrr]\ar[dllllll]\\
&&&a_0 \ar[dl]\ar[dr]&&&& a_1 \ar[dl]\ar[dr]&&&& a_2 \ar[dl]\ar[dr]&&&& a_k\ar[dl] \ar@{..}[dr]\\
\rad P(i)= &&\udots &&\ddots\ar[dr]&&\udots\ar[dl]&&\ddots\ar[dr]&&\udots\ar[dl]&&\cdots\ar[dr]&&\udots\ar[dl]&& \\
&&&&&b_1 &&&& b_2 &&&& b_k &&&&\\
&&&&&\vdots &&&& \vdots &&&& \vdots &&&&
}
\]

Let $f\in \Hom_B(\rad P(j), \rad P(i))$ be nonzero.  Then $f=(f_x)_{x\in Q_0}$ is a collection of linear maps that map each vector space at vertex $x$ in the representation $\rad P(j)$ to the vector space at $x$ in the representation $\rad P(i)$.  Note that by Proposition~\ref{prop schurian}, the dimension of the vector space at each vertex of the radicals $\rad P(i), \rad P(j)$ is at most one.  Therefore, each $f_x$ is multiplication by a scalar.   Moreover, the linear maps $\varphi_{\alpha}, \varphi'_{\alpha}$ on the arrows $\alpha\in Q$ in the representations $\rad P(i), \rad P(j)$ respectively are identity.  

Since $f$ is nonzero, then the restriction of $f$  to the top of $\rad P(j)$ is also nonzero.  Hence, there exists some $a_r$ such that $f_{a_r}\not=0$, and without loss of generality we may assume that $f_{a_r}=1$.  Since $f$ is a map of quiver representations, and, by the structure of the representations depicted above,  $\varphi_{\alpha}, \varphi'_{\alpha}$ equal 1 for all arrows $\alpha$ in the support of $\rad P(i), \rad P(j)$, respectively,  it follows that the map $f$ is also 1 on every vertex along the paths
\[b_r \leftarrow \dots \leftarrow a_r\to \dots \to b_{r+1}.\]
Similarly, since $f_{b_{r+1}}=f_{b_r}=1$, then we conclude that the map $f$ is 1 on every vertex along the paths $a_{r+1}\to \dots \to b_{r+1}$ and $a_{r-1}\to \dots \to b_{r}$.  In particular, $f_{a_{r+1}}=f_{a_{r-1}}=1$.  Continuing in this way we conclude that $f_{a_r}=1$ for all $r$.  Hence, $f$ maps via identity the top of $\rad P(j)$ to $\rad P(i)$.  It is easy to see that such $f$ factors through the projective $P(j)$ so $f: \rad P(j) \hookrightarrow P(j)\to \rad P(i)$.  This shows that $f$ is zero in $\underline{\Hom}_B(\rad P(j), \rad P(i))$.  Therefore, $\underline{\Hom}_B(\rad P(j), \rad P(i))=0$ which completes the proof.  
\end{proof}

\subsection{Construction of the checkerboard polygon}
Given a quiver $Q$ that satisfies Definition~\ref{def Q} and of total weight $2N$, we provided a combinatorial construction of a checkerboard polygon $\mathcal{S}$ of size $2N$ in \cite[Section 3.1]{SS}. That construction only depends on the structure of the quiver $Q$ and a priori does not take into account the representation theory of the corresponding dimer tree algebra $B$.  At the time, we were only able to show that our construction yields a model for the syzygy category of $B$ in the special case where every chordless cycle in $Q$ is of length 3. By Theorem~\ref{thm algo}, we now know that the syzygy category of $B$ is always modeled by a polygon $S$ of size $2N$. In this section, we construct the checkerboard pattern of that polygon  using the representation theory of $B$, and we show that it agrees with the combinatorial construction in \cite{SS}.

First, we will need the following result.  By Theorem~\ref{thm algo}, every indecomposable object in $\scmp\,B$ corresponds to a 2-diagonal in $S$.  By Lemma~\ref{lem:rad}, each radical is an indecomposable non-projective syzygy, and we denote by $\rho(i)$  the 2-diagonal corresponding to $\rad P(i)$.  The following result is a reformulation of Proposition~\ref{prop:barrow} in the geometric setting.   Recall that $R$ denotes the clockwise rotation in $S$.  

\begin{prop}\label{prop:cross}
If $i\to j$ is a boundary arrow in $Q$, then the 2-diagonals $R\rho(i)$ and $\rho(j)$ do not cross.  
\end{prop}

\begin{proof}
By Proposition~\ref{prop:barrow}, if $i\to j$ is a boundary arrow then the syzygies $\Omega\, \rad P(i)$ and $\rad P(j)$ do not have extensions in either direction in $\scmp\,B$.  From Lemma~\ref{lem left to right}, we know that extensions in $\scmp\,B$ correspond to crossings in the polygon, therefore $R\rho(i)$, the 2-diagonal corresponding to $\Omega\, \rad P(i)$, and $\rho(j)$ do not cross. 
\end{proof}

\begin{construction}\label{const}(Checkerboard polygon)
Given the algebra $B$ we construct a polygon $S$ together with a certain checkerboard pattern. We place line segments in the plane corresponding to 2-diagonals $\rho(i)$ and then add boundary vertices and segments of $S$.  

{\it Step 1: Construct the checkerboard patter.}  Start with a boundary arrow $i\to j$ in $Q$, then by \cite[Proposition 3.32]{SS} the corresponding 2-diagonals $\rho(i)$ and $\rho(j)$ cross in $S$.  Hence, we draw two line segments in the plane labeled $\rho(i)$ and $\rho(j)$ that cross.  Moreover, by Proposition~\ref{prop:cross} the clockwise rotation $R\rho(i)$ of $\rho(i)$ and the line $\rho(j)$ do not cross.   Hence, an endpoint of $\rho(i)$ and an endpoint of $\rho(j)$ are connected by a boundary segment of $S$ as in Figure~\ref{fig:fig1} on the left.  We also shade the triangular bounded region on the boundary of $S$. We label the lines in such a way that when going around the shaded region in the counterclockwise direction starting on the boundary edge, one passes on $\rho(i)$ first and then on $\rho(j)$, see Figure~\ref{fig:fig1}.

Next, since $i\to j$ is a boundary arrow there exists a unique chordless cycle 
\[\xymatrix@C=9pt{i\ar[r] & j\ar[r] & k \ar[r] & \dots \ar[r] & s\ar@/_9pt/[llll]}\]
containing this arrow.  For every new vertex in this cycle $k, \dots, s$, draw a line segment $\rho(k), \dots, \rho(s)$ such that two lines cross if and only if there is an arrow between their corresponding vertices.  Moreover, we draw these lines so that they bound an interior polygonal region with as many edges as vertices in the cycle.  Finally, we shade the region enclosed by these lines, see Figure~\ref{fig:fig1} on the right.

\begin{figure}
\centerline{{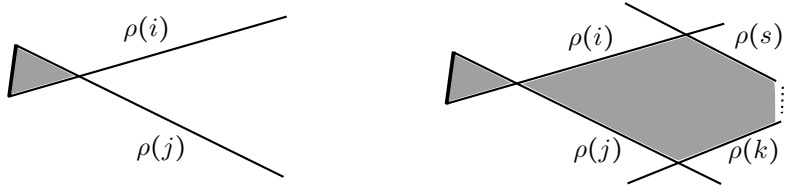}}
\caption{Construction~\ref{const} of the checkerboard pattern.}
\label{fig:fig1}
\end{figure}

If the arrow $s\to i$ is a boundary arrow, then again by Proposition~\ref{prop:cross}, we draw a boundary segment of $S$ between the two neighboring endpoints of $\rho(s)$ and $\rho(i)$ and shade the resulting triangular region on the boundary.  Otherwise,  $s\to i$ is an interior arrow, and then there exists a unique other cordless cycle containing this arrow. Similarly to the above, we add an interior shaded region whose edges correspond to the vertices in the cordless cycle.   Continue in this way until all the cycles and boundary arrows of $Q$ are exhausted, see for example Figure~\ref{fig:fig2}.

\begin{figure}
\begin{minipage}[c]{0.3\linewidth}
\[\xymatrix@R=10pt@C=10pt{
i\ar[dd]_{\color{red} 4} & s\ar[l]_{\color{red} 3}\\
&&l\ar[ul]_{\color{red} 2}\\
j\ar[r]^{\color{red} 8}&k\ar[d]^{\color{red} 5}\ar[ur]_{\color{red} 1}\\
r\ar[u]^{\color{red} 7}&t\ar[l]^{\color{red} 6}
}\]
\end{minipage}\hfill
\begin{minipage}[c]{0.7\linewidth}
\centerline{{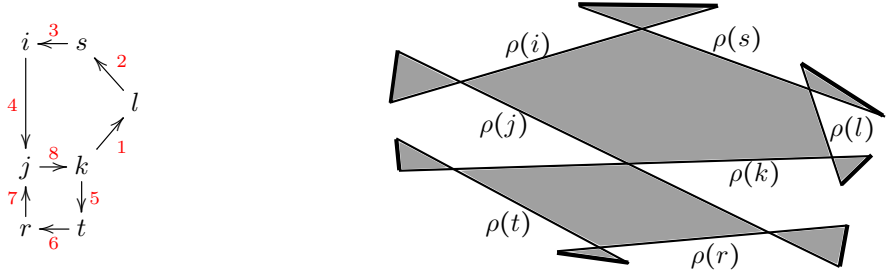}}
\end{minipage}
\caption{An example of a quiver and its checkerboard pattern after completing Step 1 of Construction~\ref{const}.}
\label{fig:fig2}
\end{figure}

Now, all 2-diagonals corresponding to radicals are complete, it remains to construct the rest of the boundary of $S$.

{\it Step 2: Orient the 2-diagonals.}  If we were to label the vertices of $S$ from 1 to $2N$ in order, then the two endpoints of each $\rho(i)$ will have opposite parities, since $\rho(i)$ is a 2-diagonal.  Therefore, we can choose to orient $\rho(i)$ so that it goes from an odd to an even labeled endpoint as follows. 
Choose an endpoint of $\rho(i)$ on the boundary of $S$ and label it odd.  Orient the line $\rho(i)$ away from this endpoint. Label the other endpoint of $\rho(i)$ even.   The odd labeled endpoint of $\rho(i)$ lies in a shaded triangular region that contains a boundary segment of $S$ and another diagonal $\rho(j)$.  Label the remaining boundary vertex of this region as even.  Then the other endpoint of $\rho(j)$ has an odd label, and we orient $\rho(j)$ from its odd to its even endpoint.  For an example see Figure~\ref{fig:fig3} on the left. 

By construction, these triangular shaded regions on the boundary correspond to boundary arrows in $Q$. Since every vertex connects to exactly two boundary arrows, it follows that the odd labeled endpoint of $\rho(j)$ also lies in a such a region together with another diagonal $\rho(r)$.  Label this endpoint of $\rho(r)$ even and orient the diagonal towards it.  Continuing in this way, we walk around the boundary of the quiver $i, j, r, \dots $ until we reach $i$ again.  Since every vertex lies on the boundary of $Q$, this determines the orientation of every radical line and labels the vertices of $S$.

\begin{figure}
\centerline{\scalebox{.9}{{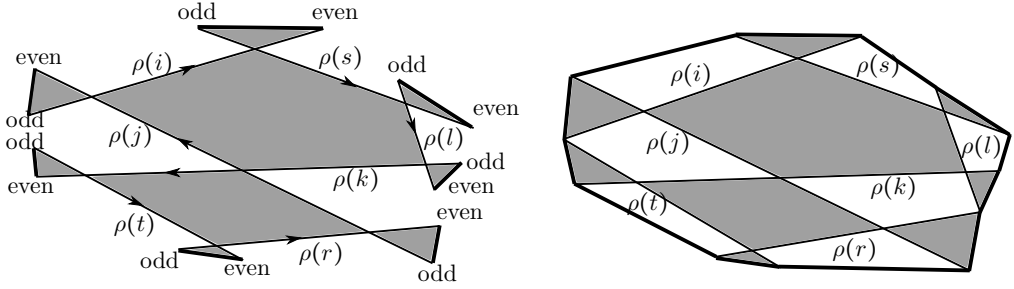}}}
\caption{The checkerboard pattern after completing Steps 2 and 3 respectively of Construction~\ref{const} of the quiver in Example~\ref{fig:fig2}.}
\label{fig:fig3}
\end{figure}

{\it Step 3: Construct the boundary of $S$.}  Every two shaded triangular regions next to each other are separated by a white region $W$.  Consider two vertices $x,y$ of $W$ on the boundary of $S$, see Figure~\ref{fig:fig4}.  If the labels of $x$ and $y$ have the same parity, then we identify them. Otherwise, if the labels of $x$ and $y$ have opposite parity, then we add a boundary segment of $S$ that connects these endpoints.   This gives a polygon with a checkerboard pattern, see Figure~\ref{fig:fig3} on the right for an example.  Note that this is the unique minimal way to close up the white regions to get a polygon given the parity of the vertices.
\end{construction} 

\begin{figure}
\centerline{\scalebox{.9}{{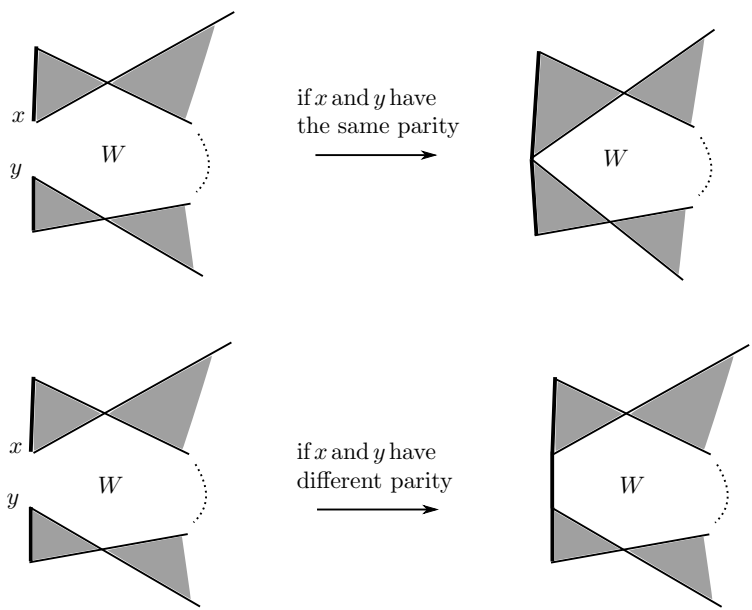}}}
\caption{Closing white regions in Step 3 of Construction~\ref{const}.}
\label{fig:fig4}
\end{figure}

\bigskip

The next result is the main theorem of this section. It says that the above construction gives an alternate way to obtain the checkerboard polygon of \cite{SS}.

Let $B=\textup{Jac}(Q,W)$ be a dimer tree algebra of total weight $2N$. Denote by $\phi\colon\diags\to\scmp\,B$ the triangle equivalence of Theorem~\ref{thm algo}, where $S$ is a polygon with $2N$ vertices. For $i\in Q_0$, let $\rho (i)=\phi^{-1}(\rad P(i))$ be the 2-diagonal in $S$ corresponding to the radical of $P(i)$. We call the $\rho(i)$  \emph{radical lines}.
\begin{thm}\label{thm poly}
With the notation above, the positions of the radical lines $\rho(i)$ in the polygon $S$ are uniquely determined up to rotation. The resulting checkerboard pattern on $S$ is the same as the checkerboard pattern defined in \cite{SS}.
\end{thm}
\begin{proof}
 The relative position of each radical line in the construction~\ref{const} is determined as soon as the first two radical lines $\rho(i), \rho(j)$ are drawn in step 1, see the left picture in Figure~\ref{fig:fig1}. Furthermore, the positions of $\rho(i)$ and $\rho(j)$ are determined up to rotation.
 
 Now we show that the construction \ref{const} produces the same checkerboard pattern as the one in \cite{SS}.  Recall that the construction in \cite{SS} realizes the pattern using the medial graph of the twisted dual graph of the quiver $Q$. In \ref{const} each chordless cycle of $Q$ gives rise to an interior shaded region in $S$. In \cite{SS}, this region is realized by the medial graph of the dual graph of the chordless cycle, see the pictures on the left in Figure~\ref{fig medial}.
\begin{figure}[htbp]
\begin{center}
\scriptsize\scalebox{1.2}{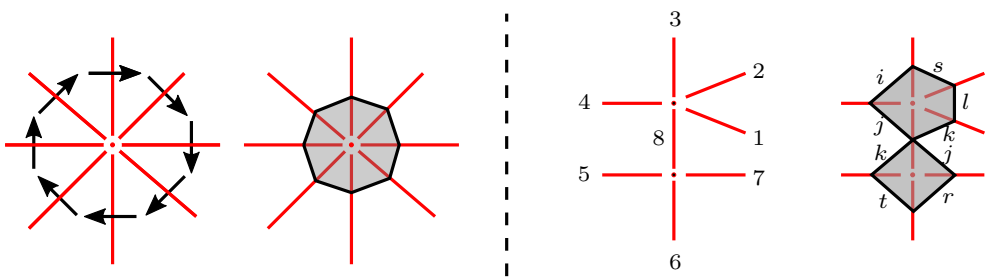}
\caption{Proof of Theorem~\ref{thm poly}. On the left, the shaded region of a chordless cycle of $Q$ is realized in the medial graph of the the dual graph of $Q$. On the right, the twisted dual graph of the quiver in Figure~\ref{fig:fig2} and the corresponding shaded regions of the medial graph.}
\label{fig medial}
\end{center}
\end{figure}

If two chordless cycles share an arrow $j\to k$ as in Figure~\ref{fig:fig2}, the corresponding shaded regions meet at the intersection of the radical lines $\rho(j)$ and $\rho(k)$. In the quiver $Q$, the two chordless cycles always have opposite orientations, one clockwise and the other counterclockwise, because they share the arrow $j\to k$.  However, in the checkerboard pattern, both cycles are recovered by going around the corresponding shaded regions in the same counterclockwise direction.

This is exactly the reason why in \cite{SS} the dual graph is twisted. For example, in the situation of Figure~\ref{fig:fig2}, the  twisted dual graph and its medial graph are shown in the two pictures on the right in Figure~\ref{fig medial}. The twisting occurs along the edge corresponding to the arrow $j\to k$ (labeled 8 in the figure), and it is the reason why the boundary arrows $5$ and $7$ switch their relative positions when going from the dual graph to the twisted dual graph. The medial graph is shown in the rightmost picture in Figure~\ref{fig medial}. It agrees with the interior shaded regions in Figure~\ref{fig:fig2}.

This shows that the two constructions produce the same checkerboard patterns in the interior of the polygon. The fact that the checkerboard patterns also agree at the boundary follows from the uniqueness in step 3 of the construction~\ref{const}.
\end{proof}

\begin{corollary}
The checkerboard pattern obtained in construction \ref{const} has the following properties.
\begin{itemize}
\item [(a)] The interior shaded regions correspond to the chordless cycles in $Q$, the boundary shaded regions to the boundary arrows in $Q$ and the crossings between the radical lines to the arrows in $Q$.
\item [(b)] The orientations of the radical lines are such that the segments that bound a shaded region are oriented in the same way (either clockwise or counterclockwise) and the segments that bound  a white region have alternating orientation.
\item [(c)] Every white region has either a single vertex on the boundary or a boundary segment.
\item [(d)] Every white region has an even number of sides.
\item[(e)] If $\za$ is a boundary arrow and $\cals_\za$ the corresponding shaded region then the cycle path $\mathfrak{c}(\za)$ and the cocycle path $\overline{\mathfrak{c}}(\za)$ of $\za$ are given by the bounding edges of the white regions adjacent to $\cals_\za$.
\end{itemize}
\end{corollary}
\begin{proof} 
These results were shown in \cite{SS}.
Part (a) is Lemma~3.1, (b) is Lemma~3.24, (c) is Lemma~3.11, (d) is Lemma~3.20, and (e) is Remark~3.14  in that paper.
\end{proof}

\section{Proof of Theorem~\ref{thm main}}\label{sect 3}
Theorem~\ref{thm algo} implies that there exists a triangle equivalence 
$ \phi \colon\textup{Diag}(S)\to \scmp\,B,$
where $S$ is a polygon with $2N$ vertices and $2N$ is the total weight of the quiver $Q$. Let $\cals$ be the checkerboard polygon of $Q$ constructed in \cite{SS}. Then $\cals$ and $S$ have the same size. By Theorem~\ref{thm poly}, there exists a triangle equivalence $\pi\colon \diag\to \textup{Diag}(S)$ such that 
$F=\phi\circ\pi \colon \diag \to \scmp\,B$ 
maps the radical line $\rho(i)$ to $\rad P(i)$, for all $i$.
In fact, up to rotation, $\pi$ is simply given by forgetting the checkerboard pattern.

Moreover, $F(R(\zg))=\zO\,F(\zg)$, because $\zO$ is the negative shift in $\scmp\,B$ by \cite{KR} and $R$ is the negative shift in $\diag$ by \cite{BM}. 
Therefore we also have $F(R^2(\zg))=\zO^2 F(\zg)$, and hence
$F(R^2(\zg))=\tau^{-1} F(\zg)$, because  $\tau^{-1}=\zO^2$ in $\scmp\,B$ by \cite{KR}.

It remains to show that the projective resolutions in $\scmp\,B$ are determined by the crossing patterns in $\diag$. Let $M\in\scmp\,B$ be indecomposable, and let $\zg=F^{-1}(M)$ denote the corresponding 2-diagonal in $ \diag$. Let 
\[\xymatrix{P_1\ar[r]^f&P_0\ar[r]&M\ar[r]&0}\] be a minimal projective presentation of $M$. Then part (a) of Proposition~\ref{prop:ext} implies that $P_0=\oplus_x P(x)$, where the sum is  over all $x\in Q_0$ such that $\Ext^1_{\scmp\,B} (M, \rad P(x))\ne 0$. Now using the equivalence $F$ together with Lemma~\ref{lem left to right}(a), we see that 
$P_0=\oplus_x P(x)$, where the sum is over all $x\in Q_0$ such that $\rho(x)$ crosses $\zg$ from right to left.
Similarly, the parts (b) of Proposition~\ref{prop:ext} and Lemma~\ref{lem left to right} imply that 
$P_1=\oplus_y P(y)$, where the sum is over all $y\in Q_0$ such that $\rho(x)$ crosses $\zg$ from  left to right.
Thus the morphism $f\colon P_1\to P_0$ satisfies the conditions of the map $f_\zg$ in the statement. 

Finally, the fact that $F$ maps 2-pivots to irreducible morphisms and meshes to Auslander-Reiten triangles follows directly from \cite{BM}.

\appendix 
\section{}
Here we give a proof of the following well-known result.
\begin{prop}
 Let $Q$ be the quiver given by a single chordless cycle of length $N$ and  $B=\textup{Jac}(Q,W)$ be its dimer tree algebra. Let $S$ be a polygon with $2N$ vertices. Then $\scmp \, B\cong \diags$.   
\end{prop}
\begin{proof}
 The algebra $B$ is cluster-tilted of type $\mathbb{D}_N$. Thus $B=\End_\calc(T)$, where $\calc$ is the cluster category of type $\mathbb{D}_N$ and $T$ is a cluster-tilting object in $\calc$, see \cite{BMRRT}. The cluster category is equivalent to the category of tagged arcs in the punctured polygon $P$ with $N$ vertices \cite{S}. Such an arc is determined by its endpoints if we agree that the arcs $(a,b)$ that connect a boundary point $a$ to a boundary point $b$ go counterclockwise around the puncture, see Figure~\ref{figchi} for an example, and the arcs  that go from a boundary point $a$ to the puncture $p$ come in pairs, a plain arc $(a,p)$ and a notched arc $(a,p)^\bowtie$.
 
   The cluster-tilting object $T$ corresponds to the triangulation consisting of the $N$ plain arcs $(a,p)$ that are incident to the puncture, and by \cite{BMR}, the module category $\textup{mod}\,B$ is equivalent to the category of all arcs that are not in $T$.  
 
 Moreover the algebra $B$ is self-injective and therefore every $B$-module is a syzygy. Thus the stable category $\scmp\, B$ is equal to the stable module category. In the geometric model, the projective modules correspond to the notched arcs $(a,p)^\bowtie$ incident to the puncture $p$. Thus the stable category $\scmp\, B$  corresponds to the category of all arcs $(a,b)$ where $a,b$ are boundary vertices in $P$. We denote this set of arcs by $\textup{Arcs}_\partial P$. Thus
 \[\textup{Arcs}_\partial P =\{(a,b)\mid 1\le a,b\le N,\  a\ne b,\  a+1\nequiv b \pmod N\}.\]
 The irreducible morphisms are given by pivots and the Auslander-Reiten quiver is illustrated in Figure~\ref{fig App} in the case $N=5$.
 \begin{figure}
\begin{center}
\scalebox{0.8}{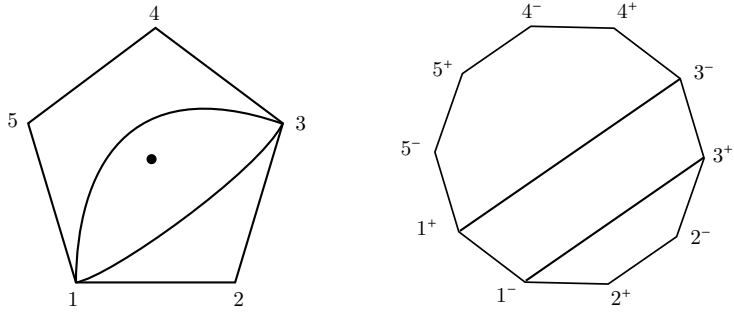}
\caption{The the arcs $(1,3)$ and $(3,1)$ in the punctured 5-gon $P$ and their images $(1^-,3^+)$ and $(3^-, 1^+)$ under the bijection $\chi$ the the 10-gon $\cals$.}
\label{figchi}
\end{center}
\end{figure}

 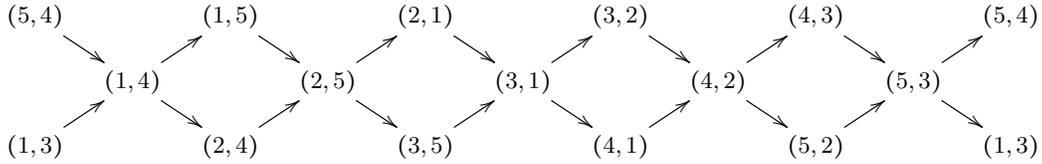
\begin{figure}
\begin{center}
\small\[\xymatrix@R10pt@C10pt{
(5,4)\ar[dr]&&(1,5)\ar[dr]&&(2,1)\ar[dr]&&(3,2)\ar[dr]&&(4,3)\ar[dr]&&(5,4)
\\
&(1,4)\ar[ru]\ar[rd]&&(2,5)\ar[ru]\ar[rd]&&(3,1)\ar[ru]\ar[rd]&&(4,2)\ar[ru]\ar[rd]&&(5,3)\ar[ru]\ar[rd]
\\
(1,3)\ar[ru]&&(2,4)\ar[ru]&&(3,5)\ar[ru]&&(4,1)\ar[ru]&&(5,2)\ar[ru]&&(1,3)}
\]
\caption{The Auslander-Reiten quiver of $\scmp\,B$ in terms of the arcs that start and end at the boundary in a punctured 5-gon (top) and the Auslander-Reiten quiver of $\diag$ (bottom).}
\label{fig App}
\end{center}
\end{figure}

Let $S$ be the polygon with $2N$ vertices labeled $1^+,1^-,2^+,2^-,\dots,n^+,n^-$ in counterclockwise order around the boundary, see Figure~\ref{figchi}. Then there is a bijection  
\[
\begin{array}{ccc}
\chi\colon\textup{Arcs}_\partial P&\longrightarrow &
\textup{Diag}(S)\\
 (a,b)&\longmapsto& (a^-,b^+).
\end{array}
\]
Indeed, $\chi$ is well-defined, since $(a^-,b^+)$ is a 2-diagonal, $\chi$ is clearly injective, and if $(a^-,b^+)$ is a 2-diagonal in $S$, then $ b>a+1$ and hence there exists an arc $(a,b) \in \textup{Arcs}_\partial P$ and thus $\chi$ is surjective.

Furthermore, the map $\phi$ induces an isomorphism between the Auslander-Reiten quivers of $\scmp\,B$ and $\textup{Diag}(S)$. This completes the proof.
\end{proof}

\end{document}